\documentclass[11pt]{article}

\usepackage{amsmath}
\usepackage{amssymb}
\usepackage{amsthm}
\usepackage{enumitem}
\usepackage{titlesec}
\usepackage{etoolbox}

\makeatletter
\patchcmd{\ttlh@hang}{\parindent\z@}{\parindent\z@\leavevmode}{}{}
\patchcmd{\ttlh@hang}{\noindent}{}{}{}
\makeatother

\usepackage{comment}

\usepackage{indentfirst}

\newtheorem{thm}{Theorem}[section]
\newtheorem{lemma}[thm]{Lemma}
\newtheorem{prop}[thm]{Proposition}

\theoremstyle{definition}
\newtheorem{defin}[thm]{Definition}
\newtheorem{ex}[thm]{Example}

\newcommand{\N}[3]{\mathcal{#1}_{#2}^{#3}}

\numberwithin{equation}{section}

\author{\uppercase{\footnotesize{Riccardo W. Maffucci}}}
\title{\normalsize{\uppercase{\bf{Nodal intersections for arithmetic random waves against a surface}}}}
\date{}
\newcommand{\Addresses}{{
  \bigskip
  \footnotesize

  R.W.~Maffucci, \textsc{Mathematical Institute, University of Oxford, Woodstock Road Oxford OX2 6GG, UK}\par\nopagebreak
  \texttt{riccardo.maffucci@maths.ox.ac.uk}
}}

\titleformat{\section}
{\Large\scshape\centering\bf}{\thesection}{1em}{}
\titleformat{\subsection}
{\large\scshape\bf}{\thesubsection}{1em}{}

\begin{document}
\maketitle
\begin{abstract}
Given the ensemble of random Gaussian Laplace eigenfunctions on the three-dimensional torus (`3d arithmetic random waves'), we investigate the $1$-dimensional Hausdorff measure of the nodal intersection curve against a compact regular toral surface (the `nodal intersection length'). The expected length is proportional to the square root of the eigenvalue, times the surface area, independent of the geometry.

Our main finding is the leading asymptotic of the nodal intersection length variance, against a surface of nonvanishing Gauss-Kronecker curvature. The problem is closely related to the theory of lattice points on spheres: by the equidistribution of the lattice points, the variance asymptotic depends only on the geometry of the surface.
\end{abstract}

\tableofcontents

\section{Introduction}
\subsection{Nodal intersections for Laplace eigenfunctions}
The {\em nodal set} of a function is its zero-locus. Several recent works (e.g. \cite{totzel,cantot,elhtot}) studied the number of intersections between the nodal lines of eigenfunctions and a fixed reference curve ({\em nodal intersections} on `generic' surfaces). It is expected that in many situations the nodal intersections number obeys the bound $\ll\sqrt{E}$, where $E>0$ is the eigenvalue.

On the three-dimensional standard flat torus $\mathbb{T}^3=\mathbb{R}^3/\mathbb{Z}^3$, the non-zero Laplace eigenvalues, or {\em energy levels}, are of the form $4\pi^2 m$, $m\in S^{(3)}$, where
\begin{equation*}
S^{(3)}:=\{0<m: \ m=a_1^2+a_2^2+a_3^2,\  a_i\in\mathbb{Z}\}.
\end{equation*}
Let \begin{equation}
\label{E}
\mathcal{E}=\mathcal{E}_m:=\{\mu=(\mu^{(1)},\mu^{(2)},\mu^{(3)})\in\mathbb{Z}^3 : ({\mu^{(1)}})^2+({\mu^{(2)}})^2+({\mu^{(3)}})^2=m\}
\end{equation}
be the set of all lattice points on the sphere of radius $\sqrt{m}$. The (complex-valued) Laplace eigenfunctions may be written as \cite{brnoda}
\begin{equation*}
G(x)=G_m(x)=\sum_{\mu\in\mathcal{E}}
c_{\mu}
e^{2\pi i\langle\mu,x\rangle},
\qquad x\in\mathbb{T}^3,
\end{equation*}
with $c_\mu$ Fourier coefficients.

A natural number $m$ is representable as a sum of three integer squares if and only if $m\neq 4^l(8k+7)$, for $k,l$ non-negative integers \cite{harwri, daven1}. The lattice points number, or equivalently the number of ways to express $m$ as a sum of three squares will be denoted
\begin{equation*}
N=N_m:=|\mathcal{E}|=r_3(m).
\end{equation*}
Under the assumption $m\not\equiv 0,4,7 \pmod 8$, one has
\begin{equation}
\label{totnumlp3}
(\sqrt{m})^{1-\epsilon}
\ll
N
\ll
(\sqrt{m})^{1+\epsilon}
\end{equation}
for all $\epsilon>0$ \cite[section 1]{bosaru}. The condition $m\not\equiv 0,4,7 \pmod 8$, ensuring $N_m\to\infty$, is natural \cite[section 1.3]{ruwiye}: indeed, if $m\equiv 7 \pmod 8$, then $m\not\in S^{(3)}$; on the other hand,
\begin{equation*}
\mathcal{E}_{4m}=\{2\mu : \mu\in\mathcal{E}_{m}\}
\end{equation*}
(see e.g. \cite[\S 20]{harwri}), hence it suffices to consider energies $4\pi^2 m$ for $m$ taken up to multiples of $4$ (see section \ref{seclp} for details).

It is known \cite{cheng1, rudwi2} that the nodal set
\begin{equation*}
\mathcal{A}_G:=\{x\in\mathbb{T}^3 : G(x)=0\}
\end{equation*}
is a union of smooth surfaces, possibly together with a set of lower dimension (curves and points). Let $\Sigma\subset \mathbb{T}^3$ be a fixed smooth reference surface. Consider the intersection
\begin{equation*}
\mathcal{A}_G \cap \Sigma,
\end{equation*}
in the limit $m\to\infty$. Bourgain and Rudnick \cite{brnoda, brgafa} found that, for $\Sigma$ real-analytic, with nowhere zero Gauss-Kronecker curvature, there exists $m_\Sigma$ such that for every $m\geq m_\Sigma$, the following hold:
\begin{itemize}
\item
the surface $\Sigma$ is not contained in the nodal set of any eigenfunction $G_m$ \cite[Theorem 1.2]{brnoda};
\item
the $1$-dimensional Hausdorff measure of the intersection has the upper bound
\begin{equation}
\label{BRthm}
h_1(\mathcal{A}_{G_m} \cap \Sigma)<C_\Sigma\cdot\sqrt{m}
\end{equation}
for some constant $C_\Sigma$ \cite[Theorem 1.1]{brgafa};
\item
for every eigenfunction $G_m$,
\begin{equation*}
\mathcal{A}_{G_m} \cap \Sigma\neq\emptyset
\end{equation*}
\cite[Theorem 1.3]{brgafa}.
\end{itemize}

\subsection{Arithmetic random waves}
\label{secarw}
We work with an ensemble of {\em random} Gaussian Laplace toral eigenfunctions (`arithmetic random waves' \cite{orruwi, rudwi2, krkuwi})
\begin{equation}
\label{arw}
F(x)={F_m^{(3)}}(x)=\frac{1}{\sqrt{N}}
\sum_{(\mu^{(1)},\mu^{(2)},\mu^{(3)})\in\mathcal{E}}
a_{\mu}
e^{2\pi i\langle\mu,x\rangle},
\qquad x\in\mathbb{T}^3,
\end{equation}
with eigenvalue $4\pi^2m$, where $a_{\mu}$ are complex standard Gaussian random variables \footnote{Defined on some probability space $(\Omega,\mathcal{F},\mathbb{P})$, where $\mathbb{E}$ denotes
the expectation with respect to $\mathbb{P}$.} (i.e., we have $\mathbb{E}[a_{\mu}]=0$ and $\mathbb{E}[|a_{\mu}|^2]=1$), independent save for the relations $a_{-\mu}=\overline{a_{\mu}}$ (so that $F(x)$ is real valued). Several recent works \cite{ruwiye,maff3d,benmaf,cammar} study the fine properties of the nodal set of \eqref{arw}. Our object of investigation is the following.
\begin{defin}
	\label{def}
Let $\Sigma$ be a fixed compact {\em regular} \footnote{i.e. a smooth $2$-dimensional submanifold of $\mathbb{T}^3$, possibly with boundary. For background on regular surfaces, see e.g. \cite{docarm, serne2}.} toral surface, of (finite) area $A$. Assume that $\Sigma$ admits a smooth normal vector locally. We define the {\bf nodal intersection length for arithmetic random waves}
\begin{equation}
\label{L}
\mathcal{L}=\mathcal{L}_m:=h_1(\mathcal{A}_{F_m} \cap \Sigma),
\end{equation}
where $h_1$ is $1$-dimensional Hausdorff measure.
\end{defin}
We will study $\mathcal{L}$ in the high energy limit $m\to\infty$.

\subsection{Prior work on nodal intersections for random waves}
\label{prior}
Denote
\begin{equation*}
S^{(d)}:=\{0<m: \ m=a_1^2+\dots+a_d^2,\  a_i\in\mathbb{Z}\}
\end{equation*}
the set of natural numbers expressible as a sum of $d$ squares. For $m\in S^{(d)}$, let
\begin{equation*}
{\mathcal{E}_m^{(d)}}:=\{\mu\in\mathbb{Z}^d : |\mu|^2=m\}
\end{equation*}
be the set of lattice points on the sphere (or circle for $d=2$)
\begin{equation*}
\sqrt{m}\mathcal{S}^{d-1}=\{y\in\mathbb{R}^d : |y|=\sqrt{m}\},
\end{equation*}
of cardinality $N_m^{(d)}=|\mathcal{E}_m^{(d)}|$. Consider the ensemble of arithmetic random waves
\begin{equation}
\label{arwd}
F_m^{(d)}(x)=\frac{1}{\sqrt{N_m^{(d)}}}
\sum_{\mu\in{\mathcal{E}_m^{(d)}}}
a_{\mu}
e^{2\pi i\langle\mu,x\rangle}
\end{equation}
defined on the $d$-dimensional torus $\mathbb{T}^d=\mathbb{R}^d/\mathbb{Z}^d$, with eigenvalue $4\pi^2m$, and the $a_\mu$'s defined as in section \ref{secarw}.

Fix a smooth reference curve $\mathcal{C}\subset\mathbb{T}^d$ and consider the number of nodal intersections
\begin{equation}
\label{Zd}
\mathcal{Z}=\N{Z}{m}{(d)}(F)=|\{x : F(x)=0\} \cap \mathcal{C}|
\end{equation}
as $m\to\infty$. Rudnick-Wigman \cite{rudwig} and subsequently Rossi-Wigman \cite{roswig} and the author \cite{maff2d} investigated $\N{Z}{m}{(2)}$. Rudnick-Wigman-Yesha \cite{ruwiye} and subsequently the author \cite{maff3d} studied the three-dimensional analogue $\N{Z}{m}{(3)}$. We may view the nodal intersection length $\mathcal{L}_m$ as another three-dimensional analogue of $\N{Z}{m}{(2)}$: indeed, in both cases one considers nodal intersections against a linear manifold.

The expected number of nodal intersections against smooth curves $\mathcal{C}$ of length $L$ is \cite{rudwig, ruwiye}
\begin{equation}
\label{ninexpect}
\mathbb{E}[\N{Z}{m}{(d)}]=\frac{2}{\sqrt{d}}\sqrt{m}\cdot L,
\end{equation}
consistent with \eqref{BRthm}, independent of the geometry. Let $d=2$: as $m\to\infty$, the number of lattice points on $\sqrt{m}\mathcal{S}^1$ satisfies
\begin{equation*}
N_m^{(2)}\ll m^\epsilon \qquad \forall\epsilon>0,
\end{equation*}
and $N\to\infty$ for a density one sequence of energy levels. Let the curve $\mathcal{C}\subset\mathbb{T}^2$ be smooth, of nowhere zero curvature, with arc-length parametrisation given by $\gamma:[0,L]\to\mathcal{C}$. Rudnick-Wigman found the precise asymptotic behaviour \cite[Theorem 1.2]{rudwig}
\begin{equation}
\label{ninvar}
\text{Var}(\N{Z}{m}{(2)})=
(4B_{\mathcal{C}}(\mathcal{E})-L^2)
\cdot
\frac{m}{N_m}
+
O\bigg(\frac{m}{N_m^{3/2}}\bigg)
\end{equation}
where
\begin{equation}
\label{bce}
B_{\mathcal{C}}(\mathcal{E}):=
\int_{\mathcal{C}}
\int_{\mathcal{C}}
\frac{1}{N_m}
\sum_{\mu\in\mathcal{E}}
\bigg\langle\frac{\mu}{|\mu|},\dot{\gamma}(t_1)\bigg\rangle^2
\cdot
\bigg\langle\frac{\mu}{|\mu|},\dot{\gamma}(t_2)\bigg\rangle^2
dt_1dt_2.
\end{equation}
This asymptotic behaviour is non-universal: $B_{\mathcal{C}}(\mathcal{E})$ depends both on $\mathcal{C}$ and on the angular distribution of the lattice points on $\sqrt{m}\mathcal{S}^1$, as $m\to\infty$. A nice consequence of \eqref{ninexpect} and \eqref{ninvar} is that the distribution of $\mathcal{Z}$ is asymptotically concentrated at the mean value, i.e., for all $\epsilon>0$,
\begin{equation*}
\lim_{\substack{m\to\infty\\\text{s.t. } N\to\infty}}\mathbb{P}\left(\left|\frac{\N{Z}{m}{(2)}(F)}{\sqrt{2m}L}-1\right|>\epsilon\right)=0.
\end{equation*}
The leading coefficient in \eqref{ninvar} is always non-negative and bounded \cite[sections 1 and 7]{rudwig}:
\begin{equation*}
0\leq 4B_{\mathcal{C}}(\mathcal{E})-L^2\leq L^2,
\end{equation*}
though it might vanish, for instance when $\mathcal{C}$ is a circle, independent of $\mathcal{E}$.

Rossi-Wigman \cite{roswig} investigated the scenario of curves such that $4B_{\mathcal{C}}(\mathcal{E})-L^2$ vanishes universally (`static curves'). For a density one sequence of energies, the precise asymptotics of the variance in the case of static curves are \cite[Theorem 1.3]{roswig}
\begin{equation}
\label{precasym2}
\text{Var}(\N{Z}{m}{(2)})=
(16A_{\mathcal{C}}(\mathcal{E})-L^2)
\cdot
\frac{m}{4N_m^2}
\cdot
(1+o(1)),
\end{equation}
where $A_{\mathcal{C}}(\mathcal{E})$ depends both on $\mathcal{C}$ and on the limiting angular distribution of the lattice points. The leading term in \eqref{precasym2} is bounded away from zero \cite[Theorem 1.3]{roswig}.

\subsection{Statement of main results}
\label{secresults}
We now state our theorems on expectation and variance of the nodal intersection length \eqref{L}.
\begin{prop}
\label{explen}
Let $\Sigma\subset\mathbb{T}^3$ be a surface as in Definition \ref{def}, of area $A$. Then we have
\begin{equation}
\label{expect}
\mathbb{E}[\mathcal{L}]=\frac{\pi}{\sqrt{3}}\sqrt{m}\cdot A.
\end{equation}
\end{prop}
The proof of Proposition \ref{explen} will be given in section \ref{seckr}. Note that \eqref{expect} is consistent with \eqref{BRthm}, and that \eqref{expect} and \eqref{ninexpect} are of similar shape. In the statement of our next result, $\overrightarrow{n}(\sigma)$ is the unit normal vector to $\Sigma$ at the point $\sigma$.
\begin{thm}
\label{varthm}
Let $\Sigma\subset\mathbb{T}^3$ be a surface as in Definition \ref{def}, of area $A$, with nowhere vanishing Gauss-Kronecker curvature. Define
\begin{equation}
\label{I}
\mathcal{I}=\mathcal{I}_\Sigma:=\iint_{\Sigma^2}\langle\overrightarrow{n}(\sigma),\overrightarrow{n}(\sigma')\rangle^2d\sigma d\sigma'.
\end{equation}
Then we have as $m\to\infty$, $m\not\equiv 0,4,7 \pmod 8$,
\begin{equation}
\label{var}
\text{Var}(\mathcal{L})=\frac{\pi^2}{60}\frac{m}{N}\left[3\mathcal{I}-A^2+O\left(m^{-1/28+o(1)}\right)\right].
\end{equation}
\end{thm}
Theorem \ref{varthm} will be proven in section \ref{sec2mom}. Compare the expressions \eqref{bce} and \eqref{I}: while the integral $B_{\mathcal{C}}(\mathcal{E})$ depends on both the curve $\mathcal{C}\subset\mathbb{T}^2$ and on the angular distribution of lattice points on circles, the integral $\mathcal{I}$ depends on $\Sigma\subset\mathbb{T}^3$ only. This is because lattice points on spheres are {\em equidistributed} \footnote{Lattice points on {\em circles} equidistribute \cite{erdhal, equidi} for a density one sequence of energies. To the other extreme, Cilleruelo proved that there exist sequences s.t. all the lattice points lie on arbitrarily short arcs \cite[Theorem 2]{ciller}. The non-uniform densities that emerge (``attainable measures'') were partially classified in \cite{krkuwi, kurwig}.} (Linnik's problem, see section \ref{seclp}). Our next result concerns the analysis of the quantity $\mathcal{I}$.

\begin{prop}
\label{Ibds}
Let $\Sigma\subset\mathbb{T}^3$ be a surface as in Definition \ref{def}, of area $A$. The integral $\mathcal{I}$ satisfies the sharp bounds
\begin{equation}
\label{Ibounds}
\frac{A^2}{3}\leq\mathcal{I}\leq A^2,
\end{equation}
so that the leading coefficient of \eqref{var} is always non-negative and bounded:
\begin{equation*}
0\leq\frac{\pi^2}{60}(3\mathcal{I}-A^2)\leq\frac{\pi^2}{30}A^2.
\end{equation*}
\end{prop}
Proposition \ref{Ibds} will be proven in section \ref{secleadconst}. A computation shows that when $\Sigma$ is a sphere or a hemisphere, the lower bound in \eqref{Ibounds} is achieved, hence the leading term in \eqref{var} vanishes: in this case the variance is of lower order than $m/N$ (see section \ref{secleadconst} for details). As in the problem of nodal intersections against a curve on $\mathbb{T}^2$ \cite{rudwig}, the theoretical maximum of the variance leading term is achieved in the case of intersection with a manifold of identically zero curvature (straight lines in dimension $2$, planes in dimension $3$). As the case of $\Sigma$ contained in a plane is excluded by the assumptions of Theorem \ref{varthm}, the upper bound of $A^2\cdot\pi^2/30$ for the leading coefficient in \eqref{var} is a supremum rather than a maximum, as in \cite{rudwig} (see section \ref{secleadconst}).

Similarly to \cite{rudwig, ruwiye}, the above results on expectation and variance have the following consequence.
\begin{thm}
Let $\Sigma\subset\mathbb{T}^3$ be a surface as in Definition \ref{def}, of area $A$, with nowhere vanishing Gauss-Kronecker curvature. Then the nodal intersection length $\mathcal{L}$ satisfies, for all $\epsilon>0$,
\begin{equation*}
\lim_{\substack{m\to\infty \\ m\not\equiv 0,4,7 \pmod 8}}\mathbb{P}\left(\left|\frac{\mathcal{L}}{\pi\sqrt{m}A/\sqrt{3}}-1\right|>\epsilon\right)=0.
\end{equation*}
\end{thm}
\begin{proof}
Apply the Chebychev-Markov inequality together with \eqref{expect} and \eqref{var}.
\end{proof}

\subsection{Outline of the proofs and plan of the paper}
\label{secout}
Throughout we apply many ideas of \cite{rudwi2, krkuwi, rudwig, ruwiye}. The arithmetic random wave $F^{(d)}:\mathbb{T}^d\to\mathbb{R}$ \eqref{arwd} is a {\em random field}. The number of nodal intersections $\mathcal{Z}^{(d)}$ \eqref{Zd} against a {\em curve} are the zeros of a {\em process}, which is the restriction of the random wave $F$ to the curve \cite{rudwig, roswig, ruwiye}. For a smooth process $p:T\to\mathbb{R}$ defined on an appropriate parameter set $T\subset\mathbb{R}$, moments of the number of zeros may be computed, under certain assumptions, via Kac-Rice formulas \cite{azawsc, cralea}.

More generally, given a smooth random field $P:T\subset_{\text{open}}\mathbb{R}^n\to\mathbb{R}^{n'}$, let $\mathcal{V}$ be the Hausdorff measure 
of its zero set. When $n-n'=0$, $\mathcal{V}$ is the number of zeros; when $n-n'=1$, $\mathcal{V}$ is the nodal length of $P$; when $n-n'=2$, $\mathcal{V}$ is the nodal area, and so forth. Only the case $n\geq n'$ is interesting, since otherwise the zero set of $P$ is almost surely \footnote{The expression `almost surely', or for short `a.s.', means `with probability $1$'.} empty. One may compute, under appropriate assumptions, the moments of $\mathcal{V}$ by means of Kac-Rice formulas \cite[Theorems 6.2, 6.3, 6.8 and 6.9]{azawsc}. The latter formulas, however, are not applicable to our problem, as the following example illustrates.
\begin{ex}
\label{ex1}
Assume that the surface $\Sigma\subset\mathbb{T}^3$ is the {\em graph of a differentiable function}, in the sense that it admits everywhere the parametrisation
\begin{align}
\notag
\gamma:U\subset\mathbb{R}^2&\to\Sigma,
\\\label{monge}
(u,v)&\mapsto(u,v,h(u,v)),
\end{align}
with $h\in C^2(U)$. We restrict $F$ to $\Sigma$, and obtain the random field 
$f:U\subset\mathbb{R}^2\to\mathbb{R}$,
\begin{equation*}
f(u,v):=F(\gamma(u,v))=\frac{1}{\sqrt{N}}
\sum_{\mu\in\mathcal{E}}
a_{\mu}
e^{2\pi i\langle\mu,(u,v,h(u,v))\rangle}.
\end{equation*}
The zero line of $f$ is not necessarily (isometric to) the nodal intersection curve $\{x\in\mathbb{T}^3 : F(x)=0\}$: rather, it is isometric to the {\em projection} of the nodal intersection curve onto the domain $U$ of the parametrisation $\gamma$. Therefore, the application of Kac-Rice formulas for $f$ yields the moments of the length of the projected curve (see \cite[Theorem 11.3]{azawsc}), not of the intersection curve itself: this is in marked contrast with what happens in the case of the nodal intersections number \cite{rudwig, roswig, ruwiye}.
\end{ex}

Our approach to the problem begins with the derivation of Kac-Rice formulas for a random field defined on a {\em surface}, which is done in section \ref{seckr} (also see \cite[Theorem 5.3]{leten1} and \cite[Theorems 4.1 and 4.4]{leten2}). Applying the Kac-Rice formula for the expectation (Proposition \ref{krexplen} to follow), we will prove Proposition \ref{explen}.

For the nodal intersection length variance we apply Proposition \ref{krvarlen} and subsequently reduce our problem to estimating the second moment of the {\em covariance function} 
\begin{equation}
\label{rintro}
r(\sigma,\sigma'):=\mathbb{E}[F(\sigma)F(\sigma')]
\end{equation}
and of its first and second order derivatives. We thus develop an {\bf approximate Kac-Rice formula}. To state it, we need some extra notation: first, $M:=4\pi^2m/3$. Moreover, $X,X',Y,Y'(\sigma,\sigma')$ are appropriate $2\times 2$ matrices, depending on $r$, its derivatives, and on the surface $\Sigma$ (see Definition \ref{XYdef}).
\begin{prop}[Approximate Kac-Rice formula]
\label{approxKR}
Let $\Sigma\subset\mathbb{T}^3$ be a surface as in Definition \ref{def}, with nowhere vanishing Gauss-Kronecker curvature. Then we have
\begin{multline*}
\text{Var}(\mathcal{L})=M\left[
\iint_{\Sigma^2}
\left(
\frac{1}{8}r^2
+\frac{tr(X)}{16}
+\frac{tr(X')}{16}
+\frac{tr(Y'Y)}{32}
\right)
d\sigma d\sigma'
\right.
\\
\left.
+O_\Sigma\left(\frac{1}{m^{11/16-\epsilon}}\right)\right].
\end{multline*}
\end{prop}
The proof of this result takes up the whole of section \ref{secappkr}. Our problem of computing the nodal intersection length variance \eqref{var} is thus reduced to estimating the second moment of the covariance function $r$ \eqref{rintro} and of its various first and second order derivatives, which is carried out in section \ref{sec2mom}. The error term in Proposition \ref{approxKR} comes from bounding the fourth moment of $r$ and of its derivatives: this is done in section \ref{sec4mom}. To study the second and fourth moments of $r$, one needs to understand various properties of the lattice point set $\mathcal{E}$ \eqref{E}, covered in section \ref{seclp}. In section \ref{secleadconst}, we study the leading term of the nodal intersection length variance \eqref{var}, and establish Proposition \ref{Ibds}. Appendix \ref{appa} is dedicated to proving several auxiliary lemmas.

\subsection{Future directions}
As discussed in section \ref{prior}, for the problem of nodal intersections against a curve in two dimensions, Rossi-Wigman \cite{roswig} investigated the case of static curves, and found the precise asymptotic behaviour of the variance, for a density one sequence of energies. It would be interesting to find, if any exist, families of `static surfaces' (other than spheres and hemispheres) satisfying
\begin{equation*}
\mathcal{I}=\frac{A^2}{3},
\end{equation*}
and study the variance asymptotics for these.

In the setting of higher-dimensional standard flat tori $\mathbb{T}^d$ for $d\geq 4$, to the best of our knowledge the only result available in the literature is the expectation \eqref{ninexpect}. Higher dimensions and intersections against higher-dimensional toral submanifolds are likely to be interesting problems at the interface of number theory and geometry.

\subsection{Acknowledgements}
The author worked on this project mainly during his PhD studies, under the supervision of Igor Wigman. The author is very grateful to Igor for suggesting this very interesting problem, and for his guidance, insightful remarks and corrections. Many thanks to Maurizia Rossi for helpful discussions. Many thanks to an anonymous referee for helpful corrections on a previous version of this manuscript. The author was supported by a Graduate Teaching Scholarship, Department of Mathematics, King's College London. The author was supported by the Engineering \& Physical Sciences Research Council (EPSRC) Fellowship EP/M002896/1 held by Dmitry Belyaev.

\section{Kac-Rice formulas for random fields defined on a surface}
\label{seckr}
\subsection{Background}
Consider a random field $P$ defined on a parameter set $T\subset_{\text{open}}\mathbb{R}^n$ and taking values in $\mathbb{R}$. \footnote{For an underlying probability space $(\Omega,\mathcal{F},\mathbb{P})$, we define $P:\Omega\times T\to\mathbb{R}$.} We always assume that the $P(t)$, called the {\em realisations} or {\em sample paths} of our random field, are almost surely continuous in $t$. A random field $P=(P_t)_t$, $t\in T$, is {\em Gaussian} if, for all $k=1,2\dots$ and every $t_1,\dots,t_k\in T$, the random vectors
\begin{equation*}
(P(t_1),\dots,P(t_k)),
\end{equation*}
called {\em finite-dimensional distributions} of $P$, are multivariate Gaussian. A centred (i.e., mean $0$) Gaussian field may be completely described by its covariance function (see Kolmogorov's Theorem \cite[section 3.3]{cralea} or \cite[section 1.2]{azawsc}).

The arithmetic random wave \eqref{arw} is a centred Gaussian {\em stationary} random field, in the sense that its covariance function
\begin{equation*}
r(x,y)
=
\frac{1}{N}\sum_{\mu\in\mathcal{E}} e^{2\pi i\langle\mu,x-y\rangle},
\qquad
x,y\in\mathbb{T}^3
\end{equation*}
depends on the difference $x-y$ only. The restriction of $F$ to $\Sigma$
\begin{equation*}
F(\sigma)=\frac{1}{\sqrt{N}}
\sum_{\mu\in\mathcal{E}}
a_{\mu}
e^{2\pi i\langle\mu,\sigma\rangle},
\qquad
\sigma\in\Sigma
\end{equation*}
is a centred Gaussian random field, with unit variance and covariance function
\begin{equation}
\label{rF}
r(\sigma,\sigma')
=
\frac{1}{N}\sum_{\mu\in\mathcal{E}} e^{2\pi i\langle\mu,\sigma-\sigma'\rangle},
\qquad
\sigma,\sigma'\in\Sigma.
\end{equation}

As mentioned in section \ref{secout}, for a {\em process} $p$ (i.e., a random field with a one-dimensional parameter set) satisfying appropriate assumptions, moments of the number of zeros may be computed via {\bf Kac-Rice formulas} \cite{azawsc, cralea, adltay}. More generally, for a random field defined on $\mathbb{R}^n$, there exist under certain conditions Kac-Rice formulas computing the moments of the Hausdorff measure $h_{n-1}$ of its ($n-1$-dimensional) zero set \cite[Chapter 6]{azawsc} \cite[Section 3]{canhan}. For a real-valued random field $P:\Sigma\to\mathbb{R}$ defined on a smooth {\em surface} $\Sigma$, consider its nodal length,
\begin{equation*}
h_1\{\sigma\in\Sigma : P(\sigma)=0\}.
\end{equation*}
The formulas of \cite{azawsc} are not applicable to this case, since in particular $\Sigma$ is not a set of full measure in $\mathbb{R}^3$ (also recall Example \ref{ex1}). Given a random field $\mathcal{X}:\mathbb{R}^3\to\mathbb{R}$ and a surface $\Sigma\subset\mathbb{R}^3$ satisfying appropriate assumptions, we derive Kac-Rice formulas for the first and second moments of the nodal length of $P=\left.\mathcal{X}\right|_\Sigma$ (also see \cite[Theorem 5.3]{leten1} and \cite[Theorems 4.1 and 4.4]{leten2}).

\subsection{Co-area formula}
\noindent
Firstly, we require a general (i.e. concerning manifolds) version of the {\em co-area formula}.
\begin{prop}[General co-area formula {\cite[Theorem 3.1]{fede59}}]
\label{coarea}
Let $X$ and $Y$ be separable $C^1$-smooth Riemannian manifolds 
with
\begin{equation*}
\dim(X)=n\geq k=\dim(Y),
\end{equation*}
and $\varphi:X\to Y$ be a Lipschitz map. We denote $J\varphi$ the Jacobian \footnote{The square root of the sum of the squares of the $k\times k$ minors of the $k\times n$ matrix
	\begin{equation*}
		\left(\frac{\partial\varphi_i}{\partial x_j}\right)_{i,j}
	\end{equation*}
\cite[Definition 2.10]{fede59}.
}
of $\varphi$. Then
\begin{equation*}
\int_B J\varphi(x)dh_nx
=
\int_Y h_{n-k}(B\cap \varphi^{-1}\{y\})dh_ky
\end{equation*}
whenever $B$ is an $h_n$-measurable subset of $X$, and consequently
\begin{equation*}
\int_X g(x)J\varphi(x)dh_nx
=
\int_Y
\left[
\int_{\varphi^{-1}\{y\}}
g(x)dh_{n-k}x
\right]
dh_ky
\end{equation*}
whenever $g$ is an $h_n$-integrable function on $X$.
\end{prop}

Next we require the definition of {\em surface gradient}. For a differentiable map $\psi:\mathbb{R}^3\to\mathbb{R}$ and a point $x\in\mathbb{R}^3$, we employ the standard notation
\begin{equation*}
\nabla \psi(x)=\left(\frac{\partial\psi}{\partial x_1},\frac{\partial\psi}{\partial x_2},\frac{\partial\psi}{\partial x_3}\right)
\end{equation*}
for the gradient of $\psi$ in $\mathbb{R}^3$.

\begin{defin}
	\label{surfgrad}
Fix a surface $\Sigma\subset\mathbb{R}^3$ as in Definition \ref{def}. For every point $\sigma\in\Sigma$, denote $T_\sigma(\Sigma)$ the  tangent plane to the surface at $\sigma$. Given a differentiable map $G:\mathbb{R}^3\to\mathbb{R}$, consider its restriction to $\Sigma$. At each point $\sigma\in\Sigma$ we define the {\em surface gradient}
\begin{equation*}
\nabla_\Sigma G(\sigma),
\end{equation*}
projection of $\nabla G(\sigma)$ onto $T_\sigma\Sigma$.
\end{defin}
The surface gradient gives the direction of maximal variation of $G$ at $\sigma$ (for further details, see e.g. \cite[chapter 7]{adltay} or \cite[section 2.5]{docarm}). We record that $JG(\sigma)=|\nabla_\Sigma G(\sigma)|$.

\begin{prop}[Deterministic nodal length]
	\label{detnl}
Let $\Sigma\subset\mathbb{R}^3$ be a surface as in Definition \ref{def} and $G:\Sigma\to\mathbb{R}$ be a smooth map satisfying $G(\sigma)=0 \Rightarrow \nabla_{\Sigma} G(\sigma)\neq 0$ for every $\sigma\in\Sigma$. 
Denote
\begin{equation}
\label{LdefG}
\mathcal{L}=\mathcal{L}(G,\Sigma):=h_1\{\sigma\in\Sigma : G(\sigma)=0\}
\end{equation}
the zero-length of $G$, where $h_1$ is Hausdorff measure. Then we have
\begin{equation}
\label{deterministic}
\mathcal{L}=\lim_{\epsilon\to 0}\mathcal{L}_\epsilon
\end{equation}
with
\begin{equation}
\label{epslen}
\mathcal{L}_\epsilon=\mathcal{L}_\epsilon(G,\Sigma):=
\frac{1}{2\epsilon}
\int_\Sigma
\chi\left(\frac{G(\sigma)}{\epsilon}\right)
|\nabla_\Sigma G(\sigma)|d\sigma,
\end{equation}
where $\chi$ is the indicator function of the interval $[-1,1]$.
\end{prop}
We defer the proof of Proposition \ref{detnl} to Appendix \ref{appa}.

\subsection{Kac-Rice for the expectation}
\begin{prop}[Kac-Rice for the expected length]
\label{krexplen}
Let $\mathcal{X}:\mathbb{R}^3\to\mathbb{R}$ be a Gaussian random field having $C^1$ paths, and $\Sigma\subset\mathbb{R}^3$ a surface as in Definition \ref{def}. Define $\mathcal{L}(\mathcal{X},\Sigma)$ and $\mathcal{L}_\epsilon(\mathcal{X},\Sigma)$ as in \eqref{LdefG} and \eqref{epslen} respectively. Suppose that, for all $\sigma\in\Sigma$, the distribution of the random variable $\mathcal{X}(\sigma)$ is non-degenerate. Moreover, assume that $\mathcal{L}_\epsilon$ is uniformly bounded, and that, for every $\sigma\in\Sigma$, the quantity
\begin{equation}
\label{bddfunc}
\frac{1}{2\epsilon}\mathbb{E}\left[\chi\left({\frac{\mathcal{X}(\sigma)}{\epsilon}}\right)
\cdot
|(\nabla_\Sigma \mathcal{X})(\sigma)|\right]
\end{equation}
is bounded as a function of $\sigma$ uniformly in $\epsilon$. Then we have
\begin{equation}
\label{krexpleneqn}
\mathbb{E}[\mathcal{L}]=\int_\Sigma K_{1;\Sigma}(\sigma)d\sigma
\end{equation}
where $K_{1;\Sigma}:\Sigma\to\mathbb{R}$,
\begin{equation}
\label{K1S}
K_{1;\Sigma}(\sigma):=
\phi_{\mathcal{X}(\sigma)}(0)
\cdot
\mathbb{E}\left[\left|(\nabla_\Sigma\mathcal{X})(\sigma)\right| \ \big| \ \mathcal{X}(\sigma)=0\right]
\end{equation}
and $\phi_{\mathcal{X}(\sigma)}$ is the probability density function of the random variable $\mathcal{X}(\sigma)$.
\end{prop}
\begin{proof}
We follow \cite[Proposition 4.1]{rudwi2}. By Ylvisaker's theorem \cite[Theorem 1.21]{azawsc}, one has
\begin{equation*}
\mathbb{P}(\exists\sigma\in\Sigma : \mathcal{X}(\sigma)=0, \nabla_{\Sigma} \mathcal{X}(\sigma)=0)=0.
\end{equation*}
We take expectations on both sides of \eqref{deterministic}:
\begin{equation*}
\mathbb{E}[\mathcal{L}]=\mathbb{E}\left[\lim_{\epsilon\to 0}\frac{1}{2\epsilon}
\int_\Sigma
\chi\left(\frac{\mathcal{X}(\sigma)}{\epsilon}\right)
|\nabla_\Sigma \mathcal{X}(\sigma)|d\sigma\right].
\end{equation*}

As $\mathcal{L}_\epsilon$ is uniformly bounded by assumption, we may apply the dominated convergence theorem:
\begin{equation*}
\mathbb{E}[\mathcal{L}]=\lim_{\epsilon\to 0}\frac{1}{2\epsilon}\mathbb{E}\left[\int_\Sigma
\chi\left({\frac{\mathcal{X}(\sigma)}{\epsilon}}\right)
\cdot
|(\nabla_\Sigma \mathcal{X})(\sigma)|d\sigma\right].
\end{equation*}
By Fubini's Theorem,
\begin{equation*}
\mathbb{E}[\mathcal{L}]=\lim_{\epsilon\to 0}\frac{1}{2\epsilon}\int_\Sigma
\mathbb{E}\left[\chi\left({\frac{\mathcal{X}(\sigma)}{\epsilon}}\right)
\cdot
|(\nabla_\Sigma \mathcal{X})(\sigma)|\right]d\sigma.
\end{equation*}
By assumption, the quantity $\eqref{bddfunc}$ is bounded as a function of $\sigma$ uniformly in $\epsilon$, hence we may apply the dominated convergence theorem to exchange the order of the limit and the integral over $\Sigma$:
\begin{equation*}
\mathbb{E}[\mathcal{L}]=\int_\Sigma
\lim_{\epsilon\to 0}\frac{1}{2\epsilon}
\mathbb{E}\left[\chi\left({\frac{\mathcal{X}(\sigma)}{\epsilon}}\right)\cdot|(\nabla_\Sigma \mathcal{X})(\sigma)|\right]d\sigma.
\end{equation*}
Via dominated convergence as $\epsilon\to 0$,
\begin{multline*}
\lim_{\epsilon\to 0}\frac{1}{2\epsilon}
\mathbb{E}\left[\chi\left({\frac{\mathcal{X}(\sigma)}{\epsilon}}\right)
\cdot
|(\nabla_\Sigma \mathcal{X})(\sigma)|\right]
\\=
\lim_{\epsilon\to 0}\frac{1}{2\epsilon}\int_{-\epsilon}^{\epsilon}\mathbb{E}\left[\left|(\nabla_\Sigma{\mathcal{X}})(\sigma)\right|\big|{\mathcal{X}}(\sigma)=x\right]\phi_{{\mathcal{X}}(\sigma)}(x)dx
\\=
\phi_{{\mathcal{X}}(\sigma)}(0)
\cdot
\mathbb{E}\left[\left|(\nabla_\Sigma{\mathcal{X}})(\sigma)\right|\big|{\mathcal{X}}(\sigma)=0\right]=K_{1;\Sigma}(\sigma)
\end{multline*}
hence \eqref{krexpleneqn}.
\end{proof}
Let us compare the quantity $K_{1;\Sigma}$ \eqref{K1S} to the {\em zero density function} (or {\em first intensity}) $K_1:\mathbb{R}^3\to\mathbb{R}$,
\begin{equation*}
K_1(x):=\phi_{\mathcal{X}(x)}(0)\cdot\mathbb{E}\left[|\nabla \mathcal{X}(x)|\ \big| \ \mathcal{X}(x)=0\right]
\end{equation*}
of the random field $\mathcal{X}:\mathbb{R}^3\to\mathbb{R}$: the zero density function has the gradient $\nabla$ of $\mathbb{R}^3$, in place of the surface gradient $\nabla_\Sigma$, in its definition. We will call $K_{1;\Sigma}$ the ``zero density of $\left.\mathcal{X}\right|_\Sigma$'', as a generalisation of $K_1$ to random fields defined on a manifold.

\subsection{The proof of Proposition \ref{explen}}
Recall the expression of the arithmetic random wave $F$ \eqref{arw}. The following lemma, that will be proven in appendix \ref{appa}, shows that $F$ satisfies one of the hypotheses of Proposition \ref{krexplen}.

\begin{lemma}
\label{unifbdd}
Let $F=F_m$ be an arithmetic random wave, and $\Sigma\subset\mathbb{T}^3$ a surface as in Definition \ref{def}. Then we have
\begin{equation*}
\mathcal{L}_\epsilon(F,\Sigma)\leq 18\sqrt{m},
\end{equation*}
with $\mathcal{L}_\epsilon$ as in \eqref{epslen}.
\end{lemma}

We now compute the zero density $K_{1;\Sigma}$ for arithmetic random waves.
\begin{lemma}
\label{K1expr}
Given the random field $\mathcal{X}=F$, define the function $K_{1;\Sigma}$ as in \eqref{K1S}. Then we have 
\begin{equation}
\label{K1expreqn}
K_{1;\Sigma}(\sigma)\equiv\frac{\pi}{\sqrt{3}}\sqrt{m}.
\end{equation}
\end{lemma}
\noindent
Before proving Lemma \ref{K1expr}, we will complete the proof of Proposition \ref{explen}.
\begin{proof}[Proof of Proposition \ref{explen}]
We need to show that the hypotheses of Proposition \ref{krexplen} hold for the random field $\mathcal{X}=F$. First, the non-degeneracy condition is met, as $F$ is unit variance. The boundedness of $\mathcal{L}_\epsilon(F,\Sigma)$ was shown in Lemma \ref{unifbdd}.

Since the quantity
\begin{equation*}
\frac{1}{2\epsilon}\mathbb{E}\left[\chi\left({\frac{F(\sigma)}{\epsilon}}\right)
\cdot
|\nabla F(\sigma)|\right]
\end{equation*}
is bounded as a function of $\sigma$ independent of $\epsilon$ \cite[proof of Proposition 4.1]{rudwi2} and since, clearly, $|\nabla_{\Sigma}F(\sigma)|\leq|\nabla F(\sigma)|$, we also obtain the boundedness of
\begin{equation*}
\frac{1}{2\epsilon}\mathbb{E}\left[\chi\left({\frac{F(\sigma)}{\epsilon}}\right)
\cdot
|\nabla_\Sigma F(\sigma)|\right]
\end{equation*}
in $\sigma$ independent of $\epsilon$. Substituting the expression \eqref{K1expreqn} of the zero density $K_{1;\Sigma}$ into \eqref{krexpleneqn} yields
\begin{equation*}
\mathbb{E}[\mathcal{L}]
=\int_{\Sigma}\frac{\pi}{\sqrt{3}}\sqrt{m} \ d\sigma
=\frac{\pi}{\sqrt{3}}\sqrt{m}\cdot A.
\end{equation*}
\end{proof}

\begin{proof}[Proof of Lemma \ref{K1expr}]
We write the zero density function \eqref{K1S} for the Gaussian field $F$:
\begin{equation}
\label{K1step1}
K_{1;\Sigma}(\sigma)=
\frac{1}{\sqrt{2\pi}}
\cdot
\mathbb{E}\left[\left|(\nabla_\Sigma F)(\sigma)\right| \ \big| \ F(\sigma)=0\right].
\end{equation}
Define the vector field
\begin{align}
\label{a}
\notag
a:\Sigma&\to\mathbb{T}^3
\\
\sigma&\mapsto(\nabla_\Sigma F)(\sigma).
\end{align}
Since $F(\sigma)$ and $a(\sigma)$ are independent (as $F$ has unit variance), we may rewrite \eqref{K1step1} as
\begin{equation*}
K_{1;\Sigma}(\sigma)
=\frac{1}{\sqrt{2\pi}}\cdot
\mathbb{E}[|a(\sigma)|].
\end{equation*}

One has $\nabla F(x)\sim\mathcal{N}(0,MI_3)$ for each $x\in\mathbb{T}^3$ \cite[(4.1)]{rudwi2}. Since at each $\sigma$ the surface gradient $a$ is the projection of $\nabla F$ onto $T_\sigma\Sigma$ (see Definition \ref{surfgrad}), one has
\begin{equation*}
K_{1;\Sigma}(\sigma)\equiv\frac{\sqrt{M}}{\sqrt{2\pi}}\mathbb{E}[|\hat{a}|],
\qquad
\hat{a}\sim\mathcal{N}(0,I_2),
\end{equation*}
so that universally
\begin{equation*}
K_{1;\Sigma}(\sigma)\equiv\frac{\sqrt{M}}{\sqrt{2\pi}}\frac{\sqrt{2}}{2}\sqrt{\pi}
=\frac{\pi}{\sqrt{3}}\sqrt{m}.
\end{equation*}
\end{proof}

\subsection{Kac-Rice for the second moment}
\begin{prop}[Kac-Rice for the second moment]
\label{krvarlen}
Let $\mathcal{X}:\mathbb{R}^3\to\mathbb{R}$ be a Gaussian random field having $C^1$ paths, and $\Sigma\subset\mathbb{R}^3$ a surface as in Definition \ref{def}. Define $\mathcal{L}(\mathcal{X},\Sigma)$ and $\mathcal{L}_\epsilon(\mathcal{X},\Sigma)$ as in \eqref{LdefG} and \eqref{epslen} respectively. Assume that $\mathcal{L}_\epsilon$ is uniformly bounded, and that for almost all $\sigma,\sigma'\in\Sigma$, the quantity
\begin{equation}
\label{bddfunc2}
\mathcal{L}_{\epsilon_1,\epsilon_2}(\sigma,\sigma'):=\frac{1}{4\epsilon_1\epsilon_2}
\mathbb{E}\left[
\chi\left(\frac{\mathcal{X}(\sigma)}{\epsilon_1}\right)
\chi\left(\frac{\mathcal{X}(\sigma')}{\epsilon_2}\right)
|(\nabla_\Sigma \mathcal{X})(\sigma)||(\nabla_\Sigma \mathcal{X})(\sigma')|
\right]
\end{equation}
is bounded uniformly in $\epsilon_1,\epsilon_2$. Suppose further that for almost all $\sigma,\sigma'\in\Sigma$, the distribution of the random vector
\begin{equation*}
\left(
\mathcal{X}(\sigma),
\mathcal{X}(\sigma')
\right)
\end{equation*} 
is non-degenerate. Then we have
\begin{gather*}
\mathbb{E}[\mathcal{L}^2]=\iint_{\Sigma^2}\tilde{K}_{2;\Sigma}(\sigma,\sigma')
d\sigma d\sigma'
\end{gather*}
where $\tilde{K}_{2;\Sigma}:\Sigma\times\Sigma\to\mathbb{R}$,
\begin{multline}
\label{K2tS}
\tilde{K}_{2;\Sigma}(\sigma,\sigma'):=\phi_{\mathcal{X}(\sigma),\mathcal{X}(\sigma')}(0,0)
\\\cdot
\mathbb{E}\left[\left|(\nabla_\Sigma\mathcal{X})(\sigma)\right|\cdot\left|(\nabla_\Sigma\mathcal{X})(\sigma')\right|\big| \mathcal{X}(\sigma)=\mathcal{X}(\sigma')=0\right]
\end{multline}
and $\phi_{\mathcal{X}(\sigma),\mathcal{X}(\sigma')}$ is the joint probability density function of the random vector $(\mathcal{X}(\sigma),\mathcal{X}(\sigma'))$.
\end{prop}
\begin{proof}
We follow \cite[Proposition 5.2]{rudwi2}. By Ylvisaker's theorem \cite[Theorem 1.21]{azawsc}, one has
\begin{equation*}
\mathbb{P}(\exists\sigma\in\Sigma : \mathcal{X}(\sigma)=0, \nabla_{\Sigma} \mathcal{X}(\sigma)=0)=0.
\end{equation*}
Therefore (recall \eqref{epslen}),
\begin{multline*}
\mathbb{E}[\mathcal{L}^2]=\mathbb{E}\left[\lim_{\epsilon_1,\epsilon_2\to 0}\mathcal{L}_{\epsilon_1}\mathcal{L}_{\epsilon_2}\right]
=\mathbb{E}\left[
\lim_{\epsilon_1,\epsilon_2\to 0}
\frac{1}{2\epsilon_1}
\int_\Sigma
\chi\left(\frac{\mathcal{X}(\sigma)}{\epsilon_1}\right)
|(\nabla_\Sigma\mathcal{X})(\sigma)|d\sigma
\right.
\\
\left.
\cdot
\frac{1}{2\epsilon_2}
\int_\Sigma
\chi\left(\frac{\mathcal{X}(\sigma')}{\epsilon_2}\right)
|(\nabla_\Sigma \mathcal{X})(\sigma')|d\sigma'
\right].
\end{multline*}

As $\mathcal{L}_\epsilon$ is uniformly bounded by assumption, we may apply the dominated convergence theorem:
\begin{multline*}
\mathbb{E}[\mathcal{L}^2]
=\lim_{\epsilon_1,\epsilon_2\to 0}\frac{1}{4\epsilon_1\epsilon_2}
\cdot\mathbb{E}\left[
\iint_{\Sigma^2}
\chi\left(\frac{\mathcal{X}(\sigma)}{\epsilon_1}\right)
\chi\left(\frac{\mathcal{X}(\sigma')}{\epsilon_2}\right)
\right.
\\
\left.
\cdot|(\nabla_\Sigma \mathcal{X})(\sigma)||(\nabla_\Sigma \mathcal{X})(\sigma')|
d\sigma d\sigma'
\right].
\end{multline*}
By Fubini's theorem,
\begin{multline*}
\mathbb{E}[\mathcal{L}^2]
=\lim_{\epsilon_1,\epsilon_2\to 0}\frac{1}{4\epsilon_1\epsilon_2}
\cdot
\iint_{\Sigma^2}
\mathbb{E}\left[
\chi\left(\frac{\mathcal{X}(\sigma)}{\epsilon_1}\right)
\chi\left(\frac{\mathcal{X}(\sigma')}{\epsilon_2}\right)
\right.
\\
\left.
\cdot|(\nabla_\Sigma \mathcal{X})(\sigma)||(\nabla_\Sigma \mathcal{X})(\sigma')|
\right]
d\sigma d\sigma'.
\end{multline*}

Next, we exchange the order of taking the limit and the integration over $\Sigma^2$, using the boundedness of \eqref{bddfunc2} uniformly in $\epsilon_1,\epsilon_2$ together with the dominated convergence theorem: an upper bound for almost all $\sigma,\sigma'$ is sufficient, as changing the values of a function on a set of measure zero has no impact on integrability or value of the integral of the function. Finally, as
\begin{align*}
&\lim_{\epsilon_1,\epsilon_2\to 0}\frac{1}{4\epsilon_1\epsilon_2}
\mathbb{E}\left[
\chi\left(\frac{\mathcal{X}(\sigma)}{\epsilon_1}\right)
\chi\left(\frac{\mathcal{X}(\sigma')}{\epsilon_2}\right)
|(\nabla_\Sigma \mathcal{X})(\sigma)||(\nabla_\Sigma \mathcal{X})(\sigma')|
\right]
\\&=
\phi_{\mathcal{X}(\sigma),\mathcal{X}(\sigma')}(0,0)
\cdot
\mathbb{E}\left[\left|(\nabla_\Sigma \mathcal{X})(\sigma)\right|\left|(\nabla_\Sigma \mathcal{X})(\sigma')\right|\mid \mathcal{X}(\sigma)=\mathcal{X}(\sigma')=0\right]
\\&=\tilde{K}_{2;\Sigma}(\sigma,\sigma'),
\end{align*}
we obtain
\begin{equation*}
\mathbb{E}[\mathcal{L}^2]
=\iint_{\Sigma^2}\tilde{K}_{2;\Sigma}(\sigma,\sigma')d\sigma d\sigma'
\end{equation*}
as claimed.
\end{proof}
The quantity $\tilde{K}_{2;\Sigma}$ \eqref{K2tS} will be called the ``two-point correlation function of $\left.\mathcal{X}\right|_\Sigma$'', as a generalisation of the {\em two-point function} or {\em second intensity} $\tilde{K_2}:\mathbb{R}^3\times \mathbb{R}^3\to\mathbb{R}$,
\begin{equation*}
\tilde{K_2}(x,y):=\phi_{\mathcal{X}(x),\mathcal{X}(y)}(0,0)\cdot\mathbb{E}\left[|\nabla \mathcal{X}(x)|\cdot|\nabla \mathcal{X}(y)| \ \big| \ \mathcal{X}(x)=\mathcal{X}(y)=0\right],
\end{equation*}
to random fields defined on a manifold.

Proposition \ref{krvarlen} is a Kac-Rice formula for the usual second moment, and not for the factorial second moment. This is consistent with Kac-Rice formulas for random fields defined on $\mathbb{R}^n$ \cite[p. 134]{azawsc}, as may be seen for instance by comparing \cite[Theorem 6.3]{azawsc} with \cite[Theorem 6.9]{azawsc}.

\section{Approximate Kac-Rice formula: proof of Proposition \ref{approxKR}}
\label{secappkr}
In the present section we establish the approximate Kac-Rice formula of Proposition \ref{approxKR}. We will require a few technical lemmas: the proof of these is deferred to appendix \ref{appa}.
\subsection{An expression for the (scaled) two-point correlation function $K_{2;\Sigma}$}

Recall our notation \eqref{a}
\begin{equation*}
a(\sigma):=(\nabla_\Sigma F)(\sigma)
\end{equation*}
for the surface gradient and let $a'=a(\sigma')$. One has
\begin{equation}
\label{acoo}
a_i=(\nabla F)_i-\langle\nabla F,\overrightarrow{n}\rangle n_i, \qquad\qquad i=1,2,3,
\end{equation}
where $\overrightarrow{n}=(n_1,n_2,n_3)$ is the unit normal to the surface at the point $\sigma$. At least one coordinate of $\overrightarrow{n}$, w.l.o.g. $n_3$, is non-zero hence
\begin{equation*}
a_3=-\frac{n_1}{n_3}a_1-\frac{n_2}{n_3}a_2.
\end{equation*}
We may thus write
\begin{align}
\label{k2tstep2}
\notag&\tilde{K}_{2;\Sigma}(\sigma,\sigma')
=\phi_{F(\sigma),F(\sigma')}(0,0)
\cdot
\mathbb{E}[|a(\sigma)|\cdot|a(\sigma')| \ \big| \ F(\sigma)=F(\sigma')=0]
\\\notag&=\frac{1}{2\pi\sqrt{1-r^2}}\mathbb{E}
\left[
\left\{\left((n_1^2+n_3^2)a_1^2+(n_2^2+n_3^2)a_2^2+2n_1n_2a_1a_2\right)/n_3^2\right\}^{1/2}\cdot
\right.
\\&\notag
\left.
\cdot
\left\{\left(({n'_1}^2+{n'_3}^2){a'_1}^2+({n'_2}^2+{n'_3}^2){a'_2}^2+2n'_1n'_2a'_1a'_2\right)/{n'_3}^2\right\}^{1/2} \ \big|
\right.
\\&
\left.
\big| \ F(\sigma)=F(\sigma')=0
\right].
\end{align}

In order to rewrite \eqref{k2tstep2} in a more convenient way, we will need some extra notation. Bearing in mind \eqref{acoo}, the covariance matrix of $a$ is given by $M\cdot\Omega$, where $M=4\pi^2m/3$ and
\begin{equation}
\label{omegafull}
\Omega(\sigma):=\begin{pmatrix}
n_2^2+n_3^2 & -n_1n_2 & -n_1n_3
\\ -n_1n_2 & n_1^2+n_3^2 & -n_2n_3
\\ -n_1n_3 & -n_2n_3 & n_1^2+n_2^2
\end{pmatrix}.
\end{equation}
We will denote $\Omega'=\Omega(\sigma')$. Recall the expression \eqref{rF} for the covariance function $r$.
\begin{defin}
\label{defdh}
	Define the row vector
	\begin{equation*}
	D=D(\sigma,\sigma'):=
	\left(\frac{\partial r}{\partial \sigma_1},
	\frac{\partial r}{\partial \sigma_2},
	\frac{\partial r}{\partial \sigma_3}\right)
	=\frac{2\pi i}{N}
	\sum_{\mu\in\mathcal{E}} e^{2\pi i\langle\mu,\sigma-\sigma'\rangle}\cdot\mu,
	\end{equation*}
	where for $j=1,2,3$ we have computed the partial derivatives
	\begin{equation*}
	\frac{\partial r}{\partial x_j}(x)=\frac{2\pi i}{N}\sum_{\mu\in\mathcal{E}} e^{2\pi i\langle\mu,x\rangle}\mu^{(j)},
	\qquad\qquad
	\mu=(\mu^{(1)},\mu^{(2)},\mu^{(3)}).
	\end{equation*}
	Let
	\begin{equation*}
	H=H(\sigma,\sigma'):=H_{r}
	=-\frac{4\pi^2}{N}
		\sum_{\mu\in\mathcal{E}} e^{2\pi i\langle\mu,\sigma-\sigma'\rangle}\cdot\mu^T\mu
	\end{equation*}
	be the (symmetric) Hessian matrix of $r$. We also define the matrix
	\begin{equation}
	\label{matL}
	L:=\begin{pmatrix}
	1 & 0
	\\ 0 & 1
	\\ 0 & 0
	\end{pmatrix}.
	\end{equation}
\end{defin}
Note that
	\begin{equation*}
	\left(\frac{\partial r}{\partial \sigma'_1},
	\frac{\partial r}{\partial \sigma'_2},
	\frac{\partial r}{\partial \sigma'_3}\right)
	=D(\sigma',\sigma)=-D(\sigma',\sigma)
	\end{equation*}
and that $H(\sigma',\sigma)=H(\sigma,\sigma')$.

\begin{lemma}
\label{lecovmat}
The covariance matrix $\Phi$ of the Gaussian random vector
\begin{equation*}
\left(
F(\sigma),F(\sigma'),
a_1(\sigma), 
a_2(\sigma), 
a_1(\sigma'),
a_2(\sigma')
\right)
\end{equation*}
is given by
\begin{equation*}
\Phi_{6\times 6}=
\begin{pmatrix}
1 & r & 0 & -D\Omega'L\\
r & 1 & D\Omega L & 0\\
0 & L^T\Omega D^T & ML^T\Omega L & -L^T\Omega H\Omega'L\\
-L^T\Omega'D^T & 0 & -L^T\Omega'H\Omega L & ML^T\Omega'L
\end{pmatrix},
\end{equation*}
where $M=4\pi^2 m/3$, $\Omega$ is given by \eqref{omegafull}, and $D,H,L$ are as in Definition \ref{defdh}.
\end{lemma}
The proof of Lemma \ref{lecovmat} will be given in appendix \ref{appa}. For $r(\sigma,\sigma')\neq\pm 1$, we may apply \cite[Proposition 1.2]{azawsc} to rewrite \eqref{k2tstep2} as
\begin{equation*}
\tilde{K}_{2;\Sigma}(\sigma,\sigma')
\\=\frac{1}{2\pi\sqrt{1-r^2}}
\mathbb{E}\left[
|(\tilde{W_1},\tilde{W_2})|\cdot|(\tilde{W_3},\tilde{W_4})|
\right],
\end{equation*}
where $\tilde{W}\sim\mathcal{N}(0,\tilde{\Theta})$ and $\tilde{\Theta}$ is the reduced covariance matrix
\begin{multline*}
\tilde{\Theta}=
\begin{pmatrix}
ML^T\Omega L & -L^T\Omega H\Omega'L
\\
-L^T\Omega'H\Omega L & ML^T\Omega' L
\end{pmatrix}
\\-
\frac{1}{1-r^2}
\begin{pmatrix}
L^T\Omega D^TD\Omega L & rL^T\Omega D^TD\Omega' L \\ rL^T\Omega'D^TD\Omega L & L^T\Omega'D^TD\Omega'L
\end{pmatrix}.
\end{multline*}

We rescale the random Gaussian vector $\tilde{W}$ as $W:=\frac{1}{\sqrt{M}}\tilde{W}$. It follows that
\begin{equation*}
K_{2;\Sigma}(\sigma,\sigma'):=\frac{\tilde{K}_{2;\Sigma}(\sigma,\sigma')}{M}
=\frac{1}{2\pi\sqrt{1-r^2}}
\mathbb{E}\left[
|(W_1,W_2)|\cdot|(W_3,W_4)|
\right]
\end{equation*}
(having defined the scaled two-point function $K_{2;\Sigma}$), where $W\sim\mathcal{N}(0,\Theta)$, and $\Theta$ is the scaled covariance matrix
\begin{multline*}
\Theta=
\begin{pmatrix}
L^T\Omega L & -L^T\Omega H\Omega'L/M
\\
-L^T\Omega'H\Omega L/M & L^T\Omega' L
\end{pmatrix}
\\-
\frac{1}{(1-r^2)M}
\begin{pmatrix}
L^T\Omega D^TD\Omega L & rL^T\Omega D^TD\Omega' L \\ rL^T\Omega'D^TD\Omega L & L^T\Omega'D^TD\Omega'L
\end{pmatrix}.
\end{multline*}

One has
\begin{equation*}
L^T\Omega L=
\begin{pmatrix}
n_2^2+n_3^2& -n_1n_2 \\ -n_1n_2 &n_1^2+n_3^2
\end{pmatrix},
\end{equation*}
the upper left $2\times 2$ block of $\Omega$. We may find a {\em square root} $Q$ of $(L^T\Omega L)^{-1}$, i.e. a matrix satisfying
\begin{equation}
\label{Q^2}
Q^2=(L^T\Omega L)^{-1}.
\end{equation}
For instance, we may take explicitly
\begin{equation}
\label{Q}
Q(\sigma)=
\frac{1}{n_3^2+n_3}\cdot
\begin{pmatrix}
n_1^2+n_3^2+n_3 & n_1n_2 \\ n_1n_2 & n_2^2+n_3^2+n_3
\end{pmatrix}=Q^T.
\end{equation}

Define the Gaussian random vector
\begin{equation*}
\hat{W}:=
\begin{pmatrix}
Q & 0 \\ 0 & Q'
\end{pmatrix}
W,
\end{equation*}
with the shorthand $Q':=Q(\sigma')$. We obtain
\begin{equation*}
K_{2;\Sigma}=\frac{1}{2\pi\sqrt{1-r^2}}\mathbb{E}[|(\hat{W_1},\hat{W_2})|\cdot|(\hat{W_3},\hat{W_4})|],
\end{equation*}
where
\begin{equation*}
\hat{W}\sim\mathcal{N}(0,\hat{\Theta})
\end{equation*}
and $\hat{\Theta}$ is the covariance matrix
\begin{multline*}
\hat{\Theta}=
\begin{pmatrix}
I_2 & -QL^T\Omega H\Omega'LQ'/M
\\
-Q'L^T\Omega'H\Omega LQ/M & I_2
\end{pmatrix}
\\-
\frac{1}{(1-r^2)M}
\begin{pmatrix}
QL^T\Omega D^TD\Omega LQ & rQL^T\Omega D^TD\Omega' LQ' \\ rQ'L^T\Omega'D^TD\Omega LQ & Q'L^T\Omega'D^TD\Omega'LQ'
\end{pmatrix}.
\end{multline*}
\begin{defin}
\label{XYdef}
Let $X,Y,X',Y'$ be the following $2\times 2$ matrices:
\begin{align*}
& X:=X(\sigma,\sigma')=-
\frac{1}{(1-r^2)M}QL^T\Omega D^TD\Omega LQ,
\\
& X':=X(\sigma',\sigma),
\\
& Y:=Y(\sigma,\sigma')=-\frac{1}{M}\left[QL^T\Omega\left(H+\frac{r}{1-r^2}D^TD\right)\Omega'LQ'\right],
\\
& Y':=Y(\sigma',\sigma),
\end{align*}
where $\Omega$ is given by \eqref{omegafull}, $D,H,L$ are as in Definition \ref{defdh}, $Q$ is given by \eqref{Q}, and $Q'=Q(\sigma')$.
\end{defin}
We may now rewrite
\begin{equation}
\label{hattheta}
\hat{\Theta}=I_4+
\begin{pmatrix}
X & Y \\ Y' & X'
\end{pmatrix}.
\end{equation}
\noindent
We have established the following result.
\begin{lemma}
\label{leK2}
The (scaled) $2$-point correlation function has the expression
\begin{equation*}
K_{2;\Sigma}=\frac{1}{2\pi\sqrt{1-r^2}}\mathbb{E}[|(\hat{W_1},\hat{W_2})|\cdot|(\hat{W_3},\hat{W_4})|],
\end{equation*}
where
\begin{equation*} \hat{W}=(\hat{W_1},\hat{W_2},\hat{W_3},\hat{W_4})\sim\mathcal{N}(0,\hat{\Theta}),
\end{equation*}
and $\hat{\Theta}$ is given by \eqref{hattheta}.
\end{lemma}

\subsection{Asymptotics for $K_{2;\Sigma}$}
We will need the following lemma, to be proven in appendix \ref{appa}.
\begin{lemma}
\label{XY}
The entries of $X,X',Y,Y'$ are uniformly bounded (with respect to $\sigma,\sigma'$):
\begin{equation}
\label{XYl1}
X,X',Y,Y'\ll_{\Sigma} 1.
\end{equation}
\end{lemma}
To write an asymptotic expression for the scaled two-point function, we need the Taylor expansion of a perturbed $4\times 4$ standard Gaussian matrix
\begin{equation*}
I_4+
\begin{pmatrix}
X & Y \\ Y' & X'
\end{pmatrix},
\end{equation*}
to the second order. The case where $X'=X$ and $Y'=Y$, to the fourth order, was treated in \cite[Lemma 5.1]{krkuwi} ($4\times 4$ matrix) and \cite[Lemma 5.8]{benmaf} ($6\times 6$). In \cite{krkuwi, benmaf} the expansion up to order four is needed in light of a cancellation phenomenon known as `arithmetic Berry cancellation'; in this work we need only an expansion to order two.
\begin{lemma}
\label{lemma5.1}
Suppose
\begin{equation*} \hat{W}=(\hat{W_1},\hat{W_2},\hat{W_3},\hat{W_4})\sim\mathcal{N}(0,\hat{\Theta}),
\end{equation*}
where
\begin{equation*}
\hat{\Theta}=I_4+
\begin{pmatrix}
X & Y \\ Y' & X'
\end{pmatrix}
\end{equation*}
is positive definite with real entries, and the $2\times 2$ blocks $X,X',Y,Y'$ are symmetric. 
Then
\begin{multline*}
\mathbb{E}[|(\hat{W_1},\hat{W_2})|\cdot|(\hat{W_3},\hat{W_4})|]
=\frac{\pi}{2}\left(1+\frac{tr(X)}{4}+\frac{tr(X')}{4}+\frac{tr(Y'Y)}{8}\right)
\\+O(X^2+X'^2+Y^4+Y'^4).
\end{multline*}
\end{lemma}
\noindent
The proof of Lemma \ref{lemma5.1} will be given in appendix \ref{appa}. Assuming it, we obtain the following asymptotic for the scaled two-point function $K_{2;\Sigma}$.
\begin{prop}
\label{prop4.5}
For $\sigma,\sigma'$ such that $r(\sigma,\sigma')$ is bounded away from $\pm 1$, we have the following asymptotic for the (scaled) two point correlation function:
\begin{multline*}
K_{2;\Sigma}(\sigma,\sigma')=\frac{1}{4}
\left[
1+\frac{1}{2}r^2+\frac{tr(X)}{4}+\frac{tr(X')}{4}+\frac{tr(Y'Y)}{8}
\right]
\\+O(r^4+X^2+X'^2+Y^4+Y'^4+r^2(tr(X)+tr(X')+tr(Y'Y))),
\end{multline*}
with $X,X',Y,Y'$ as in Definition \ref{XYdef} \footnote{Here and elsewhere we will use the shorthand $O(A)$ for $O(tr(A))$.}.
\end{prop}
\begin{proof}
By assumption, $r(\sigma,\sigma')$ is bounded away from $\pm 1$, hence
$\frac{1}{\sqrt{1-r^2}}=1+\frac{1}{2}r^2+O(r^4)$. We apply Lemma \ref{lemma5.1} to expand the precise expression of the two-point function given by Lemma \ref{leK2}:
\begin{multline*}
K_{2;\Sigma}(\sigma,\sigma')=\frac{1}{2\pi \sqrt{1-r^2(x)}}\cdot\mathbb{E}[|(\hat{W_1},\hat{W_2})|\cdot|(\hat{W_3},\hat{W_4})|]
\\=\frac{1}{2\pi}\left(1+\frac{1}{2}r^2+O(r^4)\right)\cdot\left[\frac{\pi}{2}
\left(1+\frac{tr(X)}{4}+\frac{tr(X')}{4}+\frac{tr(Y'Y)}{8}\right)
\right.
\\
\left.
+O(X^2+X'^2+Y^4+Y'^4)\right],
\end{multline*}
hence the result of the present proposition.
\end{proof}

\subsection{Statement of further auxiliary lemmas}
The present section is dedicated to stating lemmas needed to prove Proposition \ref{approxKR}. The proofs of the lemmas will follow in section \ref{sec4mom}.

\begin{lemma}
\label{R4}
For $k\geq 0$, we define the $k$-th moment of the covariance function $r$ \eqref{rF} on the surface $\Sigma$,
\begin{equation}
\label{Rkeq}
\mathcal{R}_k(m):=\iint_{\Sigma^2}r^k(\sigma,\sigma')d\sigma d\sigma'.
\end{equation}
Assume $\Sigma$ is of nowhere zero Gauss-Kronecker curvature. Then for every $\epsilon>0$ we have the upper bound
\begin{equation}
\label{R4bd}
\mathcal{R}_4(m)\ll\frac{1}{m^{11/16-\epsilon}}.
\end{equation}
\end{lemma}
\noindent
The proof of Lemma \ref{R4} will be given in section \ref{sec4mom}.

\begin{lemma}
\label{X2Y4}
We have the following upper bounds:
\begin{align*}
&\iint_{\Sigma^2}tr(X^2)d\sigma d\sigma',\ \iint_{\Sigma^2}tr(X'^2)d\sigma d\sigma',\ \iint_{\Sigma^2}tr(Y^4)d\sigma d\sigma',\\ &\iint_{\Sigma^2}tr(Y'^4)d\sigma d\sigma',\
\iint_{\Sigma^2}r^2tr(X)d\sigma d\sigma',\
\iint_{\Sigma^2}r^2tr(X')d\sigma d\sigma',\\
&\iint_{\Sigma^2}r^2tr(Y'Y)d\sigma d\sigma'\ll_\Sigma\frac{1}{m^{11/16-\epsilon}}.
\end{align*}
\end{lemma}
\noindent
The proof of Lemma \ref{X2Y4} will be given in section \ref{sec4mom}.

\subsection{The proof of Proposition \ref{approxKR}}
We claim that the hypotheses of Kac-Rice (Proposition \ref{krvarlen}) hold. The non-degeneracy condition is met since
\begin{equation*}
r(\sigma,\sigma')\neq\pm 1
\end{equation*}
for almost all $\sigma,\sigma'\in\Sigma$. Moreover, by Lemma \ref{unifbdd}, one has the uniform bound
\begin{equation*}
\mathcal{L}_{\epsilon_1}\mathcal{L}_{\epsilon_2}\leq 18^2m.
\end{equation*}
Thanks to \cite[Lemma 5.3]{rudwi2} and the fact that $|\nabla_{\Sigma}F(\sigma)|\leq|\nabla F(\sigma)|$, one has for almost all $\sigma,\sigma'\in\Sigma$,
\begin{multline*}
\mathcal{L}_{\epsilon_1,\epsilon_2}(\sigma,\sigma')=\frac{1}{4\epsilon_1\epsilon_2}
\mathbb{E}\left[
\chi\left(\frac{F(\sigma)}{\epsilon_1}\right)
\chi\left(\frac{F(\sigma')}{\epsilon_2}\right)
|(\nabla_\Sigma F)(\sigma)||(\nabla_\Sigma F)(\sigma')|
\right]
\\\ll\frac{m}{\sqrt{1-r^2(\sigma,\sigma')}},
\end{multline*}
where the implied constant is independent of $\epsilon_1,\epsilon_2$. Therefore, one may exchange the order of taking the limit and the integration over $\Sigma^2$ in Proposition \ref{krvarlen}. The hypotheses of Kac-Rice are thus all verified, hence
\begin{equation}
\label{snsprepre}
\mathbb{E}[\mathcal{L}^2]
=\iint_{\Sigma^2}\tilde{K}_{2;\Sigma}(\sigma,\sigma')d\sigma d\sigma'.
\end{equation}

We will show that Proposition \ref{prop4.5} applies `almost everywhere' in the sense that $r$ is small outside of a small set. Since the surface $\Sigma$ is compact and regular, one may write
\begin{equation*}
\Sigma=\bigcup_{1\leq q\leq Q}\Sigma_q=\bigcup_{1\leq q\leq Q}\gamma_q(U_q),
\end{equation*}
where $\gamma_q:U_q\subset\mathbb{R}^2\to\Sigma$, $1\leq q\leq Q$, are finitely many parametrisations as in \eqref{monge}, and the union is disjoint save for boundary overlaps. These overlaps are a finite union of smooth curves possibly together with a set of points, therefore the intersection with
\begin{equation*}
\mathcal{A}_F:=\{x\in\mathbb{T}^3 : F(x)=0\}
\end{equation*}
is a.s. a finite set of points \cite[Theorem 1.4]{ruwiye}.

For each $q$, consider the smallest rectangle $\overline{U}_q\supseteq U_q$ with sides parallel to the coordinate axes of $\mathbb{R}^2$. Partition (with boundary overlaps) $\overline{U}_q$ into small squares $\overline{U}_{q,p}$ of side length $\delta=c_0/\sqrt{m}$ for some small $c_0>0$. This means $U_q$ is the disjoint union (with boundary overlaps) of the $U_q\cap\overline{U}_{q,p}=:U_{q,p}$. Each $\gamma_q$ is bijective, thus each $\Sigma_q$ is the disjoint union of the $\Sigma_{q,p}:=\gamma_q(U_{q,p})$.

To simplify the notation, we re-label the indices $q,p$ to a single index $i$ and write $\Sigma=\cup_i\Sigma_i$. The set $\Sigma\times\Sigma$ is thus partitioned (with boundary overlaps) into regions $\Sigma_i\times\Sigma_j=:V_{i,j}$.
\begin{defin}
\label{singset}
We say the region $V_{i,j}$ is {\em singular} if there are points $\sigma\in\Sigma_i$ and $\sigma'\in\Sigma_j$ s.t. $|r(\sigma,\sigma')|>1/2$. The union of all singular regions is the {\em singular set} $S$.
\end{defin}
\begin{lemma}
\label{Sbound}
The measure of the singular set satisfies for every $k\geq 0$ the following upper bound:
\begin{equation*}
\text{meas}(S)\ll\mathcal{R}_k(m),
\end{equation*}
where $\mathcal{R}_k(m)$ is the $k$-th moment \eqref{Rkeq} of the covariance function $r$ on the surface $\Sigma$.
\end{lemma}
\begin{proof}
As $r/\sqrt{m}$ is a Lipschitz function, with constant independent of $m$, then for each singular region $V_{i,j}=\Sigma_i\times \Sigma_j$
\begin{equation*}
|r(\sigma,\sigma')|>1/4
\end{equation*}
everywhere on $V_{i,j}$, provided $c_0$ is chosen sufficiently small. An application of the Chebychev-Markov inequality now yields the statement of the present lemma.
\end{proof}

Outside of $S$ Proposition \ref{prop4.5} applies so that we may rewrite \eqref{snsprepre} as
\begin{multline}
\label{snspre}
\mathbb{E}[\mathcal{L}^2]
=M\iint_{\Sigma^2\setminus S}
\left(
\frac{1}{4}
+\frac{1}{8}r^2
+\frac{tr(X)}{16}
+\frac{tr(X')}{16}
+\frac{tr(Y'Y)}{32}
\right)d\sigma d\sigma'
\\
+m\cdot O\iint_{\Sigma^2}(r^4+X^2+X'^2+Y^4+Y'^4+r^2(tr(X)+tr(X')+tr(Y'Y)))d\sigma d\sigma'
\\+\iint_{S}\tilde{K}_{2;\Sigma}(\sigma,\sigma')d\sigma d\sigma'.
\end{multline}
To control the third summand in \eqref{snspre} we state the following auxiliary result.
\begin{lemma}
\label{bdsingl}
We have the bound
\begin{equation*}
\iint_{S}\tilde{K}_{2;\Sigma}(\sigma,\sigma')d\sigma d\sigma'
\ll
m\cdot\mathcal{R}_4(m).
\end{equation*}
\end{lemma}
The proof of Lemma \ref{bdsingl} will follow in appendix \ref{appa}. By Lemmas \ref{R4}, \ref{X2Y4}, and \ref{bdsingl} we may rewrite \eqref{snspre} as
\begin{multline*}
\mathbb{E}[\mathcal{L}^2]
=M\left[\iint_{\Sigma^2\setminus S}
\left(
\frac{1}{4}
+\frac{1}{8}r^2
+\frac{tr(X)}{16}
+\frac{tr(X')}{16}
+\frac{tr(Y'Y)}{32}
\right)d\sigma d\sigma'
\right.
\\
\left.
+O_\Sigma\left(\frac{1}{m^{11/16-\epsilon}}\right)\right].
\end{multline*}
Changing the domain of integration to $\Sigma^2$ carries an error of $m\cdot\mathcal{R}_4(m)$ thanks to Lemmas \ref{XY} and \ref{Sbound}. Subtracting the expectation squared (Proposition \ref{explen}) completes the proof of the approximate Kac-Rice formula Proposition \ref{approxKR}.

\section{Lattice points on spheres}
\label{seclp}
\subsection{Background}
To estimate the second and fourth moments of the covariance function $r$ and of its derivatives (in sections \ref{sec2mom} and \ref{sec4mom} respectively), we will need several considerations on lattice points on spheres $\sqrt{m}\mathcal{S}^2$. An integer $m$ is representable as a sum of three squares if and only if it is not of the form $4^l(8k+7)$, for $k,l$ non-negative integers \cite{harwri, daven1}. The total number of lattice points $N_m=r_3(m)$ oscillates: it is unbounded but vanishes for arbitrarily large $m$. We have the upper bound \cite[section 1]{bosaru}
\begin{equation*}
N
\ll
(\sqrt{m})^{1+\epsilon} \quad\text{for all} \  \epsilon>0.
\end{equation*}
The condition $m\not\equiv 0,4,7 \pmod 8$ is equivalent to the existence of \textit{primitive} lattice points $(\mu^{(1)},\mu^{(2)},\mu^{(3)})$, meaning $\mu^{(1)},\mu^{(2)},\mu^{(3)}$ are coprime (see e.g. \cite[section 1]{bosaru} and \cite[section 4]{ruwiye}). In this case, we have both lower and upper bounds \eqref{totnumlp3}
\begin{equation*}
(\sqrt{m})^{1-\epsilon}
\ll
N
\ll
(\sqrt{m})^{1+\epsilon}.
\end{equation*}
This lower bound is ineffective: the behaviour of $r_3(m)$ is not completely understood \cite[section 1]{bosaru}.

Given a sphere $\mathfrak{C}\subset\mathbb{R}^3$ and a point $P\in\mathfrak{C}$, we define the {\em spherical cap} $\mathcal{T}$ centred at $P$ to be the intersection of $\mathfrak{C}$ with the ball $\mathcal{B}_s(P)$ of radius $s$ centred at $P$. We will call $s$ the {\em radius of the cap}. We shall denote
\begin{equation}
\label{chi(R,s)}
\chi(\sqrt{m},s)
=\max_\mathcal{T}\#\{\mu\in\mathbb{Z}^3\cap \mathcal{T}\}
\end{equation}
the maximal number of lattice points contained in a spherical cap $\mathcal{T}\subset \sqrt{m}\mathcal{S}^2$ of radius $s$.
\begin{lemma}[Bourgain and Rudnick {\cite[Lemma 2.1]{brgafa}}]
	\label{lemmachi}
	We have for all $\epsilon>0$,
	\begin{equation*}
	\chi(\sqrt{m},s)
	\ll
	m^\epsilon\left(1+\frac{s^2}{\sqrt{m}^{1/2}}\right)
	\end{equation*}
	as $m\to\infty$.
\end{lemma}

\begin{defin}
\label{projlpdef}
Given an integer $m$ expressible as a sum of three squares, define
\begin{equation}
\label{projlp}
\widehat{\mathcal{E}}_m:=\mathcal{E}_m/\sqrt{m}\subset\mathcal{S}^{2}
\end{equation}
to be the projection of the set of lattice points $\mathcal{E}_m$ \eqref{E} on the unit sphere (cf. \cite[(1.5)]{bosaru} and \cite[(4.3)]{ruwiye}).
\end{defin}

Linnik conjectured (and proved under GRH) that the projected lattice points $\widehat{\mathcal{E}}_m$ become equidistributed as $m\to\infty$, $m\not\equiv 0,4,7 \pmod 8$. This result was proven unconditionally by Duke \cite{duke88,dukesp} and by Golubeva-Fomenko \cite{golfom} following a breakthrough by Iwaniec \cite{iwniec}. As a consequence, one may approximate a summation over the lattice point set by an integral over the unit sphere.
\begin{lemma}[cf. {\cite[Lemma 8]{pascbo}}]
\label{equidlemma}
Let $g(z)$ be a $C^2$-smooth function on $\mathcal{S}^2$. For $m\to\infty$, $m\not\equiv 0,4,7 \pmod 8$, we have
\begin{equation*}
\frac{1}{N}\sum_{\mu\in\mathcal{E}}g\left(\frac{\mu}{|\mu|}\right)
=
\int_{z\in\mathcal{S}^2}g(z)\frac{dz}{4\pi}
+
O_g\left(\frac{1}{m^{1/28-\epsilon}}\right).
\end{equation*}
\end{lemma}

Define the probability measures
\begin{equation}
\label{tau}
\tau_m=\frac{1}{N}\sum_{\mu\in\mathcal{E}}\delta_{\mu/|\mu|}
\end{equation}
on the unit sphere, where $\delta_x$ is the Dirac delta function at $x$. By the equidistribution of lattice points on spheres, the $\tau_m$ converge weak-* \footnote{The statement $\nu_i\Rightarrow\nu$ means that, for every smooth bounded test function $g$, one has $\int gd\nu_i\to\int gd\nu$.} to the uniform measure on the unit sphere:
\begin{equation}
\label{weak}
\tau_m\Rightarrow\frac{\sin(\varphi)d\varphi d\psi}{4\pi}
\end{equation}
as $m\to\infty$, $m\not\equiv 0,4,7 \pmod 8$.

\begin{defin}
For $s>0$, the \textbf{Riesz $s$-energy} of $n$ (distinct) points $P_1,\dots,P_n$ on $\mathcal{S}^2$ is defined as
\begin{equation*}
E_s(P_1,\dots,P_n):=\sum_{i\neq j}\frac{1}{|P_i-P_j|^s}.
\end{equation*}
\end{defin}
\noindent
Bourgain, Sarnak and Rudnick computed the following precise asymptotics for the Riesz $s$-energy of the projected lattice points $\widehat{\mathcal{E}}_m\subset\mathcal{S}^2$ \eqref{projlp}.
\begin{prop}[{\cite[Theorem 1.1]{bosaru}, \cite[Theorem 4.1]{ruwiye}}]
\label{asyriesz}
Fix $0<s<2$. Suppose $m\to\infty$, $m\not\equiv 0,4,7 \pmod 8$. There is some $\delta>0$ so that
\begin{equation*}
E_s(\widehat{\mathcal{E}}_m)=I(s)\cdot N^2+O(N^{2-\delta})
\end{equation*}
where
\begin{equation*}
I(s)=\frac{2^{1-s}}{2-s}.
\end{equation*}
\end{prop}

\subsection{Lemmas on lattice points on spheres}
\noindent
We will need the following lemma, that also appears in \cite{cammar}.
\begin{lemma}[cf. {\cite[Lemma 3.3]{cammar}}]
\label{1/5}
Suppose $m\to\infty$, $m\not\equiv 0,4,7 \pmod 8$. We have, for $i,j=1,2,3$, $i\neq j$, and $0\leq k\leq l$, $k+l=4$, 
\begin{equation*}
\frac{1}{m^2N}\sum_{(\mu^{(1)},\mu^{(2)},\mu^{(3)})\in\mathcal{E}}({\mu^{(i)}})^k({\mu^{(j)}})^l=
\begin{cases}
1/5+O\left(m^{-1/28+o(1)}\right) & k=0,
\\
0 & k=1,
\\
1/15+O\left(m^{-1/28+o(1)}\right)& k=2.
\end{cases}
\end{equation*}
\end{lemma}
\begin{proof}
The case $k=1$ immediately follows from the symmetries of the lattice point set $\mathcal{E}$ (cf. \cite[Lemma 2.3]{rudwi2}). Let $k=0$: by the equidistribution of lattice points on spheres (Lemma \ref{equidlemma}), we may approximate
\begin{equation}
\label{sumtoint}
\frac{1}{m^2N}\sum_{\mu\in\mathcal{E}}{\mu^{(j)}}^4
\sim
\frac{1}{4\pi}\int_{z\in\mathcal{S}^2}z_j^4dz
\end{equation}
up to an error of $m^{-1/28+o(1)}$. We use the spherical coordinates
\begin{equation*}
z=(\sin(\varphi)\cos(\psi),\sin(\varphi)\sin(\psi),\cos(\varphi)),
\qquad\quad
0\leq\varphi\leq\pi, \ 0\leq\psi\leq 2\pi.
\end{equation*}
Take $j=3$ (by the symmetry, \eqref{sumtoint} will be independent of the choice of $j$). The integral in \eqref{sumtoint} may be rewritten as
\begin{equation*}
\frac{1}{4\pi}\int_{0}^{\pi}\cos^4(\varphi)\sin(\varphi)d\varphi \int_{0}^{2\pi}d\psi=\frac{1}{5}.
\end{equation*}
Similarly, for $k=2$, Lemma \ref{equidlemma} yields
\begin{equation*}
\frac{1}{m^2N}\sum_{\mu\in\mathcal{E}}{\mu^{(i)}}^2{\mu^{(j)}}^2
\sim
\frac{1}{4\pi}\int_{z\in\mathcal{S}^2}z_i^2z_j^2dz.
\end{equation*}
The latter integral may be computed via the same method as the case $k=0$ and equals $1/15$.
\end{proof}

We denote $\mathcal{C}(4)$ the set of length $4$ {\em spectral correlations}, i.e., $4$-tuples of lattice points on spheres summing up to $0$
\begin{equation}
\label{C4}
\mathcal{C}(4):=\mathcal{C}_m(4)
=\left\{(\mu_1,\mu_2,\mu_3,\mu_4)\in \mathcal{E}_m^4 :
\mu_1+\mu_2+\mu_3+\mu_4=0
\right\}.
\end{equation}
Denote \cite[section 2.3]{brgafa}
\begin{equation*}
\kappa(R):=\max_{\Pi} \#\{\mu\in\mathbb{Z}^3 : \mu\in R\mathcal{S}^{2}\cap\Pi\}
\end{equation*}
the maximal number of lattice points in the intersection of $R\mathcal{S}^{2}\subset\mathbb{R}^3$ and any plane $\Pi$. Jarnik (see \cite{jarnik}, \cite[(2.6)]{brgafa}) found the upper bound
\begin{equation}
\label{kappa3bound}
\kappa(R)\ll R^\epsilon, \quad \forall\epsilon>0.
\end{equation}
An upper bound for $|\mathcal{C}(4)|$ may be obtained as follows. We fix two lattice points $\mu_1,\mu_2$, which may be done in $N^2$ ways. Supposing that $\mu_1+\mu_2\neq 0$, \footnote{The relation $\mu_1+\mu_2=0$ forces also $\mu_3+\mu_4=0$ and determines $N^2$ correlations.} then $\mu_1+\mu_2+\mu_3$ must lie in the intersection of two spheres centred respectively at the origin and at the point $\mu_1+\mu_2$, both of radius $\sqrt{m}$. As two spheres intersect in a circle, we obtain the upper bound
\begin{equation*}
|\mathcal{C}(4)|\ll N^2\cdot\kappa(\sqrt{m}).
\end{equation*}
By \eqref{kappa3bound} and \eqref{totnumlp3}, it now follows that
\begin{equation}
	\label{C4bd}
	|\mathcal{C}(4)|\ll N^{2+\epsilon} \qquad \forall\epsilon>0.
\end{equation}

\begin{lemma}
\label{4mom4corr}
We have the bound
\begin{equation}
\label{eqn4cor}
\sum_{\mathcal{E}^4\setminus\mathcal{C}(4)}\frac{1}{|\mu_1+\mu_2+\mu_3+\mu_4|^2}\ll N^{2+5/8+\epsilon}.
\end{equation}
\end{lemma}
\begin{proof}
This proof is inspired by the two-dimensional analogue \cite[Lemma 6.2]{rudwig}. Let $v:=\mu_1+\mu_2+\mu_3+\mu_4$: as $|v|\neq 0$, we clearly have $|v|\geq 1$. we separate the summation in \eqref{eqn4cor} over three ranges: $1\leq|v|\leq A$, $A\leq|v|\leq B$, and $|v|\geq B$, where $A=A(m)$ and $B=B(m)$ are parameters.

\underline{First range: $1\leq|v|\leq A$.} Given $v$ and $\mu_1,\mu_2$ such that $\mu_1+\mu_2\neq v$, the number of solutions $(\mu_3,\mu_4)$ to
\begin{equation*}
v=\mu_1+\mu_2+\mu_3+\mu_4
\end{equation*}
is $\ll\kappa(\sqrt{m})\ll m^\epsilon$: to see this, we use the same argument as in the bound for $|\mathcal{C}(4)|$ above. Therefore,
\begin{multline}
\label{regi1}
\sum_{1\leq|v|\leq A}\frac{1}{|\mu_1+\mu_2+\mu_3+\mu_4|^2}\ll N^{2+\epsilon}\sum_{1\leq|v|\leq A}\frac{1}{|v|^2}\ll N^{2+\epsilon}\int_{1\leq|x|\leq A}\frac{dx}{|x|^2}\\\ll N^{2+\epsilon}A.
\end{multline}

\underline{Second range: $A\leq|v|\leq B$.} We fix $\mu_1,\mu_2,\mu_3$ and write
\begin{equation}
\label{regi2'}
\sum_{A\leq|v|\leq B}\frac{1}{|\mu_1+\mu_2+\mu_3+\mu_4|^2}\ll \frac{1}{A^2}N^3\cdot|\{\mu_4 : A\leq|v|\leq B\}|.
\end{equation}
We require an estimate for the second factor on the RHS of \eqref{regi2'}. Consider the geometric picture: the vector $v$ lies on the sphere centred at $\mu_1+\mu_2+\mu_3$ of radius $\sqrt{m}$, and also in the difference set of the two balls centred at the origin of radii $B,A$. Therefore, $\mu_4$ lies on a spherical cap of radius $2B$ of a sphere of radius $\sqrt{m}$, hence the bound
\begin{equation}
\label{applychi}
|\{\mu_4 : A\leq|v|\leq B\}|
\leq\chi(\sqrt{m}, 2B)\ll m^\epsilon\left(1+\frac{B^2}{\sqrt{m}^{1/2}}\right)
\end{equation}
via Lemma \ref{lemmachi}. Replacing \eqref{applychi} into \eqref{regi2'},
\begin{equation}
\label{regi2}
\sum_{A\leq|v|\leq B}\frac{1}{|\mu_1+\mu_2+\mu_3+\mu_4|^2}\ll \frac{N^3}{A^2}\cdot m^\epsilon\left(1+\frac{B^2}{m^{1/4}}\right).
\end{equation}

\underline{Third range: $|v|\geq B$.} Here we have
\begin{equation}
\label{regi3}
\sum_{|v|\geq B}\frac{1}{|\mu_1+\mu_2+\mu_3+\mu_4|^2}\ll \frac{N^4}{B^2}.
\end{equation}
Collecting the estimates \eqref{regi1}, \eqref{regi2} and \eqref{regi3}, we obtain
\begin{equation*}
\sum_{\mathcal{E}^4\setminus\mathcal{C}(4)}\frac{1}{|\mu_1+\mu_2+\mu_3+\mu_4|^2}\ll N^{2+\epsilon} A+\frac{N^3}{A^2}\cdot m^\epsilon\left(1+\frac{B^2}{m^{1/4}}\right)+\frac{N^4}{B^2}.
\end{equation*}
Bearing in mind \eqref{totnumlp3}, the optimal choice for the parameters is $(A,B)=(N^{5/8},N^{11/16})$, so that one has the estimate
\begin{equation*}
\sum_{\mathcal{E}^4\setminus\mathcal{C}(4)}\frac{1}{|\mu_1+\mu_2+\mu_3+\mu_4|^2}\ll N^{2+5/8+\epsilon}
\end{equation*}
as claimed.
\end{proof}

\section{Second moment of $r$ and of its derivatives: proof of Theorem \ref{varthm}}
\label{sec2mom}
\noindent
The present section is dedicated to proving Theorem \ref{varthm}.
\subsection{Statement of estimates}
We recall that $A$ is the area of the surface $\Sigma$, $r$ the covariance function \eqref{rF}, and Definition \ref{XYdef} for $X,X',Y,Y'$.

\begin{lemma}
\label{2momr}
Assume $\Sigma$ is of nowhere zero Gauss-Kronecker curvature. Then we have the following estimates:
\begin{align}
\label{R2}
&\iint_{\Sigma^2}r^2d\sigma d\sigma'
=\frac{A^2}{N}+O\left(\frac{1}{m^{1-\epsilon}}\right);
\\
\label{trX}
&\iint_{\Sigma^2}tr(X)d\sigma d\sigma'
=-\frac{2A^2}{N}
+O\left(\frac{1}{m^{11/16-\epsilon}}\right);
\\
\label{trX'}
&\iint_{\Sigma^2}tr(X')d\sigma d\sigma'
=-\frac{2A^2}{N}
+O\left(\frac{1}{m^{11/16-\epsilon}}\right).
\end{align}
\end{lemma}
Lemma \ref{2momr} will be proven in section \ref{secpf2momr}. We define
\begin{equation}
\label{H}
\mathcal{H}=\mathcal{H}_\Sigma(\mathcal{E}):=\iint_{\Sigma^2}\frac{1}{N}\sum_{\mu\in\mathcal{E}}\left\langle\frac{\mu}{|\mu|},\overrightarrow{n}(\sigma)\right\rangle^2\cdot\left\langle\frac{\mu}{|\mu|},\overrightarrow{n}(\sigma')\right\rangle^2d\sigma d\sigma'.
\end{equation}
\begin{lemma}
\label{2der}
Assume $\Sigma$ is of nowhere zero Gauss-Kronecker curvature. Then we have the following estimate:
\begin{equation}
\label{trY'Y}
\iint_{\Sigma^2}tr(Y'Y)d\sigma d\sigma'
=\frac{3}{N}(A^2+3\mathcal{H})+O\left(\frac{1}{m^{11/16-\epsilon}}\right).
\end{equation}
\end{lemma}
Lemma \ref{2der} will be proven in section \ref{secpf2der}. By the equidistribution of lattice points on spheres (Lemma \ref{equidlemma}), one may express $\mathcal{H}$ in terms of $\Sigma$ only, as shown in the following result.
\begin{lemma}
\label{IH}
As $m\to\infty$, $m\not\equiv 0,4,7 \pmod 8$, one has the estimate
\begin{equation*}
\mathcal{H}=\frac{1}{15}(A^2+2\mathcal{I})+O\left(\frac{1}{m^{1/28-\epsilon}}\right).
\end{equation*}
\end{lemma}
\begin{proof}
Firstly, using Lemma \ref{1/5}, we estimate the summation
\begingroup
\allowdisplaybreaks
\begin{align}
\label{IHsteps}
\notag&\sum_{\mu\in\mathcal{E}}\left\langle\frac{\mu}{|\mu|},\overrightarrow{n}(\sigma)\right\rangle^2\cdot\left\langle\frac{\mu}{|\mu|},\overrightarrow{n}(\sigma')\right\rangle^2
\\\notag&=\frac{N}{15}\left(
3\sum_{i=1}^{3}n_i^2{n'_i}^2
+\sum_{\substack{i=1\\i\neq j}}^{3}n_i^2{n'_j}^2
+\sum_{\substack{i=1\\i\neq j}}^{3}4n_in_jn'_in'_j
+
O\left(\frac{1}{m^{1/28-\epsilon}}\right)
\right)
\\&=\frac{N}{15}\left(
1+2\langle\overrightarrow{n},\overrightarrow{n}'\rangle^2
+O\left(\frac{1}{m^{1/28-\epsilon}}\right)
\right).
\end{align}
\endgroup
By substituting \eqref{IHsteps} into \eqref{H}, we obtain
\begin{align*}
\mathcal{H}_\Sigma(\mathcal{E})&=\frac{1}{15}\iint_{\Sigma^2}\left(
1+2\langle\overrightarrow{n},\overrightarrow{n}'\rangle^2
\right)d\sigma d\sigma'+
O\left(\frac{1}{m^{1/28-\epsilon}}\right)\\&=\frac{1}{15}(A^2+2\mathcal{I})+
O\left(\frac{1}{m^{1/28-\epsilon}}\right).
\end{align*}
\end{proof}

We may now establish the asymptotics for the nodal intersection length variance.
\begin{proof}[Proof of Theorem \ref{varthm} assuming Lemmas \ref{2momr} and \ref{2der}]
On substituting the estimates \eqref{R2}, \eqref{trX}, \eqref{trX'} and \eqref{trY'Y} into the approximate Kac-Rice formula Proposition \ref{approxKR}, we obtain
\begin{multline*}
\text{Var}(\mathcal{L})=M\left[\frac{1}{8}\frac{A^2}{N}
+\frac{1}{16}\left(-\frac{2A^2}{N}\right)
+\frac{1}{16}\left(-\frac{2A^2}{N}\right)
+\frac{1}{32}\frac{3}{N}(A^2+3\mathcal{H})
\right.
\\
\left.
+O\left(\frac{1}{m^{11/16-\epsilon}}\right)\right]
=\frac{\pi^2}{24}\frac{m}{N}\left(9\mathcal{H}-A^2\right)+O\left(\frac{1}{m^{3/16-\epsilon}}\right).
\end{multline*}
The estimate for $\mathcal{H}$ given by Lemma \ref{IH} now implies \eqref{var}.
\end{proof}
In the rest of section \ref{sec2mom} we prove Lemmas \ref{2momr} and \ref{2der}. Firstly, we will require a bound for oscillatory integrals on a surface.
\begin{prop}[{\cite[section 7]{stein1}}]
\label{stein}
If $\Sigma$ is a hypersurface in $\mathbb{R}^n$ with non-vanishing Gaussian curvature, then
\begin{equation*}
\int_\Sigma e^{ix\cdot\xi}\psi(x)d\sigma(x)=O(|\xi|^{-(n-1)/2})
\end{equation*}
as $|\xi|\to\infty$, whenever $\psi\in C_0^{\infty}(\mathbb{R}^n)$.
\end{prop}
The case $n=2$ of Proposition \ref{stein} was proven by Van der Corput \cite[Proposition 2]{stein1} (see also \cite[Lemma 5.2]{rudwig}). Here we will need the case $n=3$ i.e., as $|\xi|\to\infty$,
\begin{equation*}
\int_\Sigma e^{ix\cdot\xi}\psi(x)d\sigma(x)=O(|\xi|^{-1}).
\end{equation*}

\subsection{Proof of Lemma \ref{2momr}}
\label{secpf2momr}
We square the covariance function \eqref{rF}
\begin{equation*}
r^2(\sigma,\sigma')
=
\frac{1}{N^2}\sum_{\mu,\mu'} e^{2\pi i\langle\mu-\mu',\sigma-\sigma'\rangle}
\end{equation*}
and integrate it over $\Sigma^2$: the diagonal terms equal
\begin{equation*}
\frac{1}{N^2}\iint_{\Sigma^2}\sum_{\mu}1d\sigma d\sigma'
=\frac{A^2}{N}.
\end{equation*}
We bound the off-diagonal terms by applying Proposition \ref{stein}
\begin{align*}
\frac{1}{N^2}\iint_{\Sigma^2}
\sum_{\mu\neq\mu'} e^{2\pi i\langle\mu-\mu',\sigma-\sigma'\rangle}d\sigma d\sigma'
&=\frac{1}{N^2}\sum_{\mu\neq\mu'}\left|\int_{\Sigma}e^{2\pi i\langle\mu-\mu',\sigma\rangle}d\sigma\right|^2
\\&\ll_\Sigma\frac{1}{N^2}\sum_{\mu\neq\mu'}\frac{1}{|\mu-\mu'|^2},
\end{align*}
followed by Proposition \ref{asyriesz} with $s=2-\epsilon$:
\begin{equation*}
\sum_{\mu\neq\mu'}\frac{1}{|\mu-\mu'|^2}\ll\frac{N^2}{(\sqrt{m})^{2-\epsilon}}.
\end{equation*}
This completes the proof of \eqref{R2}.

Next, we show \eqref{trX}, \eqref{trX'} being similar. The proof of the following preliminary lemma is deferred to Section \ref{sec4mom}.
\begin{lemma}
\label{prelim}
With $D,H$ as in Definition \ref{defdh} and $\Omega$ as in \eqref{omegafull}, we have the following estimates:
\begin{align}
&-\iint_{\Sigma^2}\frac{1}{M}D\Omega D^Td\sigma d\sigma'=-\frac{2A^2}{N}+O\left(\frac{1}{m^{1-\epsilon}}\right);\label{e1}
\\&\iint_{\Sigma^2}\frac{r^2}{M}D\Omega D^Td\sigma d\sigma'=O\left(\frac{1}{m^{11/16-\epsilon}}\right);\label{e2}
\\&\iint_{\Sigma^2}\frac{1}{M^2}tr\left(H\Omega H\Omega'\right)d\sigma d\sigma'=\frac{3}{N}(A^2+3\mathcal{H})+O\left(\frac{1}{m^{1-\epsilon}}\right);\label{e3}
\\&\iint_{\Sigma^2}\frac{r^2}{M^2}D\Omega H\Omega'D^Td\sigma d\sigma'=O\left(\frac{1}{m^{11/16-\epsilon}}\right).\label{e4}
\end{align}
\end{lemma}

By Definition \ref{XYdef},
\begin{equation*}
X=X(\sigma,\sigma'):=-\frac{1}{(1-r^2)M}QL^T\Omega D^TD\Omega LQ,
\end{equation*}
where the matrices $D,L$ are as in Definition \ref{defdh}, $\Omega$ as in \eqref{omegafull}, and $Q$ as in \eqref{Q}. We separate the domain of integration into the singular set $S$ of Definition \ref{singset} and its complement:
\begin{equation}
\label{s+ns}
\iint_{\Sigma^2}tr(X)d\sigma d\sigma'=\iint_{\Sigma^2\setminus S}tr(X)d\sigma d\sigma'+O(\text{meas}(S)),
\end{equation}
where we used the uniform boundedness of $X$ given by Lemma \ref{XY}. On $\Sigma^2\setminus S$ the covariance function $r$ is bounded away from $\pm 1$, so that
\begin{equation}
\label{trXa}
trX
=-\frac{1}{M}tr(QL^T\Omega D^TD\Omega LQ)
+O\left(tr\left(\frac{r^2}{M}QL^T\Omega D^TD\Omega LQ\right)\right).
\end{equation}
We have
\begin{equation}
\label{trXasy1}
tr(QL^T\Omega D^TD\Omega LQ)=D\Omega LQQL^T\Omega D^T=D\Omega D^T,
\end{equation}
where in the last equality we noted that 
\begin{equation}
\label{id}
LQ^2L^T\Omega=I_3
\end{equation}
by \eqref{Q^2}. Inserting \eqref{trXasy1} into \eqref{trXa} and integrating over $\Sigma^2\setminus S$,
\begin{multline}
\label{into}
\iint_{\Sigma^2\setminus S}tr(X)d\sigma d\sigma'
=-\frac{1}{M}\iint_{\Sigma^2\setminus S}D\Omega D^Td\sigma d\sigma'
\\+O\left(\iint_{\Sigma^2\setminus S}\frac{r^2}{M}D\Omega D^Td\sigma d\sigma'\right)
.
\end{multline}
By Definition \ref{defdh},
\begin{equation}
\label{DID}
D\Omega D^T
=-\frac{4\pi^2}{N^2}\sum_{\mu,\mu'} e^{2\pi i\langle\mu-\mu',\sigma-\sigma'\rangle}\mu\Omega\mu'^T.
\end{equation}
The uniform bound
\begin{equation}
\label{unibd}
|\mu^{(i)}|\leq\sqrt{m} \qquad i=1,2,3
\end{equation}
implies $D\Omega D^T\ll M$. We may then rewrite the main term of \eqref{into} as
\begin{equation}
\label{into2}
-\frac{1}{M}\iint_{\Sigma^2\setminus S}D\Omega D^Td\sigma d\sigma'
=-\frac{1}{M}\iint_{\Sigma^2}D\Omega D^Td\sigma d\sigma'
+O\left(\text{meas}(S)\right).
\end{equation}
We substitute \eqref{into2} into \eqref{into}, and then \eqref{into} into \eqref{s+ns} to obtain
\begin{multline}
\label{supernew}
\iint_{\Sigma^2}tr(X)d\sigma d\sigma'=-\frac{1}{M}\iint_{\Sigma^2}D\Omega D^Td\sigma d\sigma'+O\left(\iint_{\Sigma^2}\frac{r^2}{M}D\Omega D^Td\sigma d\sigma'\right)\\+O(\text{meas}(S)).
\end{multline}
Inserting the estimates \eqref{e1} and \eqref{e2} into \eqref{supernew}, and bearing in mind Lemmas \ref{Sbound} and \ref{R4} concludes the proof of \eqref{trX}, \eqref{trX'} being similar.

\subsection{Proof of Lemma \ref{2der}}
\label{secpf2der}
Similarly to the proof of Lemma \ref{2momr}, we write
\begin{equation}
\label{ys+ns}
\iint_{\Sigma^2}tr(Y'Y)d\sigma d\sigma'=\iint_{\Sigma^2\setminus S}tr(Y'Y)d\sigma d\sigma'+O(\text{meas}(S)).
\end{equation}
By Definition \ref{XYdef} and \eqref{id}, on $\Sigma^2\setminus S$ we write
\begin{multline}
\label{Y'Y}
Y'Y(\sigma,\sigma')=\frac{1}{M}\left[Q'L^T\Omega'\left(H+\frac{r}{1-r^2}D^TD\right)\Omega LQ\right]\\\cdot\frac{1}{M}\left[QL^T\Omega\left(H+\frac{r}{1-r^2}D^TD\right)\Omega'LQ'\right]
\\=\frac{1}{M^2}tr\left(H\Omega H\Omega'\right)
+O\left(\frac{r}{M^2}D\Omega H\Omega'D^T\right).
\end{multline}
We insert \eqref{Y'Y} into \eqref{ys+ns} and use the bound $tr(H\Omega H\Omega')\ll M$ (recall \eqref{unibd}) together with Lemmas \ref{Sbound} and \ref{R4} to obtain 
\begin{multline}
\label{new2}
\iint_{\Sigma^2}tr(Y'Y)d\sigma d\sigma'=\iint_{\Sigma^2}\frac{1}{M^2}tr\left(H\Omega H\Omega'\right)d\sigma d\sigma'
\\+O\left(\iint_{\Sigma^2}\frac{r^2}{M^2}D\Omega H\Omega'D^Td\sigma d\sigma'\right)
+O\left(\frac{1}{m^{11/16-\epsilon}}\right).
\end{multline}
Substituting \eqref{e3} and \eqref{e4} into \eqref{new2} yields \eqref{trY'Y}, concluding the proof of Lemma \ref{2der}.

\section{Fourth moment of $r$: proofs of Lemmas \ref{R4}, \ref{prelim}, and \ref{X2Y4}}
\label{sec4mom}
\begin{proof}[Proof of Lemma \ref{R4}]
We write the fourth power of $r$ \eqref{rF}
\begin{equation*}
r^4(\sigma,\sigma')
=
\frac{1}{N^4}\sum_{\mathcal{E}^4} e^{2\pi i\langle\mu_1+\mu_2+\mu_3+\mu_4,\sigma-\sigma'\rangle}
\end{equation*}
and substitute it into \eqref{Rkeq} with $k=4$ to obtain
\begin{equation*}
\mathcal{R}_4(m)=\frac{1}{N^4}
\sum_{\mathcal{E}^4}\iint_{\Sigma^2}e^{2\pi i\langle\sum\mu_j,\sigma-\sigma'\rangle}d\sigma d\sigma'
\end{equation*}
(we abbreviate $\sum\mu_j:=\mu_1+\mu_2+\mu_3+\mu_4$). We separate the summation over $\mathcal{E}^4$ into the set \eqref{C4} of $4$-spectral correlations $\mathcal{C}(4)$ and its complement.

The summation over the $4$-correlations is
\begin{equation*}
\sum_{\mathcal{C}(4)}\iint_{\Sigma^2}d\sigma d\sigma'
\ll_{\Sigma}|\mathcal{C}(4)|\ll N^{2+\epsilon},
\end{equation*}
where we applied \eqref{C4bd}. By Proposition \ref{stein}, the remaining summation is bounded by
\begin{equation*}
\sum_{\mathcal{E}^4\setminus\mathcal{C}(4)}\iint_{\Sigma^2}e^{2\pi i\langle\sum\mu_j,\sigma-\sigma'\rangle}d\sigma d\sigma'
\\\ll\sum_{\mathcal{E}^4\setminus\mathcal{C}(4)}\frac{1}{|\mu_1+\mu_2+\mu_3+\mu_4|^2}.
\end{equation*}
By Lemma \ref{4mom4corr}, the latter summation has the upper bound
$N^{2+5/8+\epsilon}$. It follows that
\begin{equation*}
\mathcal{R}_4(m)
\ll\frac{1}{N^4}(N^{2+\epsilon}+N^{2+5/8+\epsilon})\ll\frac{1}{m^{11/16-\epsilon}},
\end{equation*}
where we applied \eqref{totnumlp3}.
\end{proof}

\begin{proof}[Proof of Lemma \ref{prelim}]
One clearly has $\sum_{\mu}|\mu|^2=mN$, hence
\begin{equation}
\label{clearly}
\sum_{\mu}(\mu^{(i)})^2=\frac{mN}{3} \qquad
\text{ for } i=1,2,3.
\end{equation}
By the symmetry of the lattice points, one also has $\sum_{\mu}\mu^{(i)}\mu^{(j)}=0$ for any $i\neq j$. Bearing this in mind, we directly compute the diagonal terms $\mu=\mu'$ of \eqref{DID} to equal
\begin{equation*}
\frac{4\pi^2}{N^2}
\sum_{\mu}[(1-n_1^2)(\mu^{(1)})^2+(1-n_2^2)(\mu^{(2)})^2+(1-n_3^2)(\mu^{(3)})^2]
=\frac{2M}{N},
\end{equation*}
giving a contribution of
\begin{equation}
\label{cont1}
-\frac{1}{M}\iint_{\Sigma^2}\frac{2M}{N}d\sigma d\sigma'=-\frac{2A^2}{N}
\end{equation}
to \eqref{e1}. To control the off-diagonal terms $\mu\neq\mu'$ of $D\Omega D^T$, we start by applying Proposition \ref{stein} to the integrals
\begin{equation*}
\int_{\Sigma}\left(\frac{\mu}{|\mu|}\Omega\frac{\mu'^T}{|\mu'|}(\sigma)\right)e^{2\pi i\langle\mu-\mu',\sigma\rangle}d\sigma\ll_\Sigma\frac{1}{|\mu-\mu'|}
\end{equation*}
and
\begin{equation*}
\int_{\Sigma}e^{2\pi i\langle\mu'-\mu,\sigma'\rangle}d\sigma'
\\\ll_\Sigma\frac{1}{|\mu-\mu'|},
\end{equation*}
obtaining the bound
\begin{multline}
\label{cont2}
\sum_{\mu\neq\mu'}\int_{\Sigma}
\frac{4\pi^2}{N^2}\left(\frac{\mu}{|\mu|}\Omega\frac{\mu'^T}{|\mu'|}(\sigma)\right)e^{2\pi i\langle\mu-\mu',\sigma\rangle}d\sigma\int_{\Sigma}e^{2\pi i\langle\mu'-\mu,\sigma'\rangle}d\sigma'
\\\ll_\Sigma\frac{1}{N^2}\sum_{\mu\neq\mu'}\frac{1}{|\mu-\mu'|^2}\ll\frac{1}{(\sqrt{m})^{2-\epsilon}},
\end{multline}
where we also applied Proposition \ref{asyriesz}. Consolidating the estimates \eqref{cont1} and \eqref{cont2} completes the proof of \eqref{e1}. 

We now show \eqref{e2}. To treat the diagonal terms in $r^2D\Omega D^T$, we use \eqref{C4bd} and the triangle inequality to obtain
\begin{equation}
\label{cont3}
\iint_{\Sigma^2}\frac{1}{MN^4}\sum_{\mathcal{C}(4)}e^{2\pi i\langle\sum\mu_j,\sigma-\sigma'\rangle}\mu_3\Omega\mu_4^T
d\sigma d\sigma'
\ll\frac{1}{N^{2-\epsilon}}
\end{equation}
(recall the abbreviation $\sum\mu_j:=\mu_1+\mu_2+\mu_3+\mu_4$). Next, we consider the off-diagonal terms of $r^2D\Omega D^T$. Invoking Proposition \ref{stein} again,
\begin{multline}
\label{cont4}
\sum_{\mathcal{E}^4\setminus\mathcal{C}(4)}\int_{\Sigma}
\frac{1}{N^4}\left(\frac{\mu_3}{|\mu_3|}\Omega\frac{\mu_4^T}{|\mu_4|}(\sigma)\right)e^{2\pi i\langle\sum\mu_j,\sigma\rangle}d\sigma\int_{\Sigma}e^{2\pi i\langle-\sum\mu_j,\sigma'\rangle}d\sigma'
\\\ll_\Sigma\frac{1}{N^4}\sum_{\mathcal{E}^4\setminus\mathcal{C}(4)}\frac{1}{|\sum\mu_j|^2}\ll\frac{1}{m^{11/16-\epsilon}}
\end{multline}
having used Lemma \ref{4mom4corr} in the last inequality. Consolidating the estimates \eqref{cont3} and \eqref{cont4} completes the proof of \eqref{e2}.

To prove \eqref{e3}, we start with Definition \ref{defdh},
\begin{equation*}
tr\left(H\Omega H\Omega'\right)
=\frac{(4\pi^2)^2}{N^2}\sum_{\mu,\mu'}e^{2\pi i\langle\mu-\mu',\sigma-\sigma'\rangle}\mu\Omega\mu'^T\mu'\Omega'\mu^T.
\end{equation*}
One directly calculates the diagonal terms of the latter expression to equal
\begin{align}
\label{10}
&\notag\frac{(4\pi^2)^2}{N^2}\sum_{\mu}(m-\langle\mu,n\rangle^2)\cdot(m-\langle\mu,n'\rangle^2)
\\\notag&=\frac{(4\pi^2)^2}{N^2}\left[m^2N-m\sum_{\mu}\langle\mu,n\rangle^2-m\sum_{\mu}\langle\mu,n'\rangle^2+\sum_{\mu}\langle\mu,n\rangle^2\langle\mu,n'\rangle^2\right]
\\&=\frac{(4\pi^2)^2}{N^2}\left[\frac{m^2N}{3}+\sum_{\mu}\langle\mu,n\rangle^2\langle\mu,n'\rangle^2\right]
\end{align}
having used \eqref{clearly} in the last step. We control the contribution of the off-diagonal terms of $tr\left(H\Omega H\Omega'\right)$ via Propositions \ref{stein} and \ref{asyriesz}:
\begin{multline}
\label{11}
\frac{(4\pi^2)^2}{N^2}\sum_{\mu\neq\mu'}\left|\int_{\Sigma}\frac{\mu}{|\mu|}\Omega\frac{\mu'^T}{|\mu'|}e^{2\pi i\langle\mu-\mu',\sigma\rangle}d\sigma\right|^2
\ll_\Sigma\frac{1}{N^2}\sum_{\mu\neq\mu'}\frac{1}{|\mu-\mu'|^2}
\\\ll\frac{1}{(\sqrt{m})^{2-\epsilon}}.
\end{multline}
Substituting \eqref{10} and \eqref{11} into \eqref{new2} and recalling the definition of $\mathcal{H}$ \eqref{H} yields \eqref{e3}. The computation for \eqref{e4} is very similar to the preceding ones and we omit it here.
\end{proof}

\begin{proof}[Proof of Lemma \ref{X2Y4}]
We will prove in detail a few of the bounds, the remaining being similar. By squaring $X$ (Definition \ref{XYdef}) we obtain 
\begin{equation*}
X^2=\left[-
\frac{1}{(1-r^2)M}\right]^2\cdot\left(QL^T\Omega D^TD\Omega LQ\right)^2.
\end{equation*}
Therefore, outside the singular set $S$ (Definition \ref{singset}),
\begin{equation}
\label{X2}
tr(X^2)
\ll\frac{1}{M^2}(D\Omega D^T)^2
\end{equation}
where we used \eqref{id}. We may thus write
\begin{equation}
\label{trX2}
\iint_{\Sigma^2}tr(X^2)d\sigma d\sigma'=O\left(\iint_{\Sigma^2}\frac{1}{M^2}(D\Omega D^T)^2d\sigma d\sigma'\right)+O(\text{meas}(S)).
\end{equation}
The diagonal terms of the integral on the RHS in \eqref{trX2} give a contribution of
\begin{equation}
\label{cont5}
\iint_{\Sigma^2}\frac{1}{M^2N^4}\sum_{\mathcal{C}(4)}\mu_3\Omega\mu_4^T\mu_4\Omega'\mu_3^T
d\sigma d\sigma'
\ll\frac{1}{N^{2-\epsilon}}
\end{equation}
via \eqref{C4bd}. The off-diagonal terms are bounded via Proposition \ref{stein} and Lemma \ref{4mom4corr}:
\begin{multline}
\label{cont6}
\frac{1}{N^4}\sum_{\mathcal{E}^4\setminus\mathcal{C}(4)}\left|\int_{\Sigma}
\left(\frac{\mu_3}{|\mu_3|}\Omega\frac{\mu_4^T}{|\mu_4|}(\sigma)\right)e^{2\pi i\langle\sum\mu_j,\sigma\rangle}d\sigma\right|^2
\\\ll_\Sigma\frac{1}{N^4}\sum_{\mathcal{E}^4\setminus\mathcal{C}(4)}\frac{1}{|\sum\mu_j|^2}\ll\frac{1}{m^{11/16-\epsilon}}.
\end{multline}
Inserting the bounds \eqref{cont5} and \eqref{cont6} into \eqref{trX2} (and bearing in mind Lemmas \ref{Sbound} and \ref{R4}) yields the desired estimate for $tr(X^2)$. The one for $tr(X'^2)$ is proven in a similar way.

Let us now prove the bound for $tr(Y^4)$, the one for $tr(Y'^4)$ being similar. By Definition \ref{XYdef},
\begin{equation*}
Y=-\frac{1}{M}\left[QL^T\Omega\left(H+\frac{r}{1-r^2}D^TD\right)\Omega'LQ'\right]
\end{equation*}
hence on $\Sigma^2\setminus S$
\begin{equation}
\label{Y4}
tr(Y^4)\ll\frac{1}{M^4}tr(H\Omega H\Omega')^2
\end{equation}
via \eqref{id}. Therefore,
\begin{equation}
\label{trY4}
\iint_{\Sigma^2}tr(Y^4)d\sigma d\sigma'=O\left(\iint_{\Sigma^2}\frac{1}{M^4}tr(H\Omega H\Omega')^2d\sigma d\sigma'\right)+O(\text{meas}(S)).
\end{equation}
Similarly to the case of $tr(X^2)$, the diagonal contribution to \eqref{trY4} is $O(N^{-2+\epsilon})$, whereas the off-diagonal terms are bounded invoking Proposition \ref{stein} and Lemma \ref{4mom4corr} again:
\begin{multline*}
\frac{1}{N^4}\sum_{\mathcal{E}^4\setminus\mathcal{C}(4)}\left|\int_{\Sigma}
\frac{\mu_1}{|\mu_1|}\Omega\frac{\mu_2^T}{|\mu_2|}\frac{\mu_3}{|\mu_3|}\Omega\frac{\mu_4^T}{|\mu_4|}e^{2\pi i\langle\sum\mu_j,\sigma\rangle}d\sigma\right|^2
\\\ll_\Sigma\frac{1}{N^4}\sum_{\mathcal{E}^4\setminus\mathcal{C}(4)}\frac{1}{|\sum\mu_j|^2}\ll\frac{1}{m^{11/16-\epsilon}}.
\end{multline*}

Finally, we show the bound for $r^2 tr(X)$, the ones for $r^2 tr(X')$ and $r^2 tr(Y'Y)$ being similar. As in the previous computations, one writes
\begin{equation*}
\iint_{\Sigma^2}r^2tr(X)d\sigma d\sigma'=O\left(\iint_{\Sigma^2}\frac{r^2}{M}D\Omega D^Td\sigma d\sigma'\right)+O(\text{meas}(S)).
\end{equation*}
The required result now follows from \eqref{e2} and Lemmas \ref{Sbound} and \ref{R4}.
\end{proof}

\section{Study of the variance leading constant: proof of Proposition \ref{Ibds}}
\label{secleadconst}
Recall $A$ is the area of the surface $\Sigma$, $\overrightarrow{n}(\sigma)$ the unit normal at the point $\sigma\in\Sigma$, and $\mathcal{I}$ the integral \eqref{I}
\begin{equation*}
\mathcal{I}=\iint_{\Sigma^2}\langle\overrightarrow{n}(\sigma),\overrightarrow{n}(\sigma')\rangle^2d\sigma d\sigma'.
\end{equation*}
The goal of this section is to prove Proposition \ref{Ibds}, i.e., that the sharp bounds
\begin{equation*}
\frac{A^2}{3}\leq\mathcal{I}\leq A^2
\end{equation*}
hold.
\begin{lemma}
\label{Iupbd}
We have
\begin{equation*}
\mathcal{I}\leq A^2,
\end{equation*}
with equality if and only if $\Sigma$ is contained in a plane.
\end{lemma}
\begin{proof}
The integral $\mathcal{I}$ is maximised when
\begin{equation*}
\langle\overrightarrow{n}(\sigma),\overrightarrow{n}(\sigma')\rangle=\pm 1
\end{equation*}
for every $\sigma,\sigma'\in\Sigma$, i.e., when the normal vectors to the surface are all parallel.
\end{proof}

We now turn to establishing the lower bound $\mathcal{I}\geq A^2/3$. For any probability measure $\tau$ on $\mathcal{S}^2$ invariant by reflection w.r.t. the coordinate planes, define the number
\begin{equation*}
c(\tau,\Sigma):=\iint_{\Sigma^2}\int_{\mathcal{S}^{2}}\left\langle\theta,\overrightarrow{n}(\sigma)\right\rangle^2\cdot\left\langle\theta,\overrightarrow{n}(\sigma')\right\rangle^2d\sigma d\sigma'd\tau(\theta).
\end{equation*}
\begin{lemma}
\label{ctheta}
For any measure $\tau$ on $\mathcal{S}^2$ invariant by reflection w.r.t. the coordinate planes, we have
\begin{equation*}
\frac{A^2}{9}\leq c(\tau,\Sigma)\leq\frac{A^2}{3}.
\end{equation*}
\end{lemma}
Lemma \ref{ctheta} will be proven in a moment; assuming it, we may complete the proof of Proposition \ref{Ibds}.

\begin{proof}[Proof of Proposition \ref{Ibds} assuming Lemma \ref{ctheta}]
One has
\begin{equation*}
c\left(\frac{\sin(\varphi)d\varphi d\psi}{4\pi},\Sigma\right)=\frac{A^2+2\mathcal{I}}{15}
\end{equation*}
hence via Lemma \ref{ctheta} we obtain the lower bound in \eqref{Ibounds}:
\begin{equation*}
\mathcal{I}\geq\frac{A^2}{3}.
\end{equation*}
The upper bound $\mathcal{I}\leq A^2$ in \eqref{Ibounds} has already been proven in Lemma \ref{Iupbd}.
\end{proof}

\begin{proof}[Proof of Lemma \ref{ctheta}]
We write
\begin{equation}
\label{AMQM2}
c(\tau,\Sigma)=\int_{\mathcal{S}^{2}}q(\theta,\Sigma)^2d\tau(\theta),
\end{equation}
where for $\theta\in\mathcal{S}^2$ we have defined
\begin{equation*}
q(\theta,\Sigma):=\int_\Sigma\langle\theta,\overrightarrow{n}(\sigma)\rangle^2d\sigma.
\end{equation*}
Let $\mathfrak{R}$ be the intersection of the unit sphere and the first octant
\begin{equation*}
\mathfrak{R}:=\{(x,y,z)\in\mathbb{R}^3 : x,y,z>0\}\cap\mathcal{S}^{2}.
\end{equation*}
We complete $\theta$ to an orthonormal basis $(\theta,\theta',\theta'')$ such that $\theta,\theta',\theta''$ lie in distinct octants, hence
\begin{equation}
\label{AMQM3}
3 c(\tau,\Sigma)
=8\int_\mathfrak{R}\left[q(\theta,\Sigma)^2+q(\theta',\Sigma)^2+q(\theta'',\Sigma)^2\right]d\tau(\theta)
\end{equation}
via \eqref{AMQM2} and the symmetries of $\tau$. As $\theta,\theta',\theta''\in\mathcal{S}^2$ are pairwise orthogonal, we have
\begin{equation*}
q(\theta,\Sigma)+q(\theta',\Sigma)+q(\theta'',\Sigma)=A
\end{equation*}
so that by Cauchy-Schwartz (or the AM-QM inequality)
\begin{equation}
\label{AMQM1}
\frac{A^2}{3}\leq q(\theta,\Sigma)^2+q(\theta',\Sigma)^2+q(\theta'',\Sigma)^2\leq A^2.
\end{equation}
Inserting the two inequalities \eqref{AMQM1} into \eqref{AMQM3} yields the claim of the present lemma.
\end{proof}

Let us compare the conditions to obtain the vanishing of the variance leading term in the two- and three-dimensional settings. In the former, this occurs for certain subsets of circles \cite[Proposition 7.3]{rudwig} and more generally, for families of static curves \cite[appendix F]{roswig}. We were able to find specific examples of surfaces s.t. the variance leading term vanishes (e.g. spheres and hemispheres): it would be interesting to find, if it exists, a more general family of surfaces satisfying this condition.

\appendix

\section{Proofs of auxiliary results}
\label{appa}
\noindent
In this appendix, we prove several auxiliary propositions and lemmas.
\begin{proof}[Proof of Proposition \ref{detnl}]
	We follow the approach of \cite[Lemma 3.1]{rudwi2}. Applying Proposition \ref{coarea} with $X=\Sigma$, $Y=\mathbb{R}$,
	$\varphi=G$, 
	and
	\begin{equation*}
	g(\sigma):=\frac{1}{2\epsilon}\chi\left(\frac{G(\sigma)}{\epsilon}\right)
	\end{equation*}
	yields the equality
	\begin{equation}
	\label{coareaapp}
	\int_\Sigma \frac{1}{2\epsilon}\chi\left(\frac{G(\sigma)}{\epsilon}\right)
	|\nabla_\Sigma G(\sigma)|d\sigma
	=
	\int_\mathbb{R}
	\left[
	\int_{G(\sigma)=y}
	\frac{1}{2\epsilon}\chi\left(\frac{G(\sigma)}{\epsilon}\right)
	dh_1\sigma
	\right]
	dy.
	\end{equation}
	Substituting \eqref{coareaapp} into \eqref{epslen}, we obtain
	\begin{equation*}
	\mathcal{L}_\epsilon
	=
	\frac{1}{2\epsilon}
	\int_{-\epsilon}^{\epsilon}
	h_1\{\sigma : G(\sigma)=y\}dy.
	\end{equation*}
	
	By assumption, $G$ satisfies $G(\sigma)=0 \Rightarrow \nabla_\Sigma G(\sigma)\neq 0$ for every $\sigma\in\Sigma$, hence the function $y\mapsto h_1(G^{-1}\{y\})$ is continuous at $y=0$, so that, by the fundamental theorem of calculus,
	\begin{equation*}
	\lim_{\epsilon\to 0}\mathcal{L}_\epsilon
	=
	\lim_{\epsilon\to 0}\frac{1}{2\epsilon}
	\int_{-\epsilon}^{\epsilon}
	h_1(\{\sigma : G(\sigma)=y\})dy
	=
	h_1(G^{-1}(0))=\mathcal{L},
	\end{equation*}
	proving \eqref{deterministic}.
\end{proof}

\begin{proof}[Proof of Lemma \ref{unifbdd}]
By the definition \eqref{epslen},
\begin{equation*}
\mathcal{L}_\epsilon(F,\Sigma)
=
\frac{1}{2\epsilon}
\int_\Sigma
\chi\left(\frac{F(\sigma)}{\epsilon}\right)
|\nabla_\Sigma F(\sigma)|d\sigma.
\end{equation*}
As the surface gradient $\nabla_\Sigma F$ is the projection of $\nabla F$ on $T_\sigma\Sigma$,
\begin{equation}
\label{unifbddeqn}
\mathcal{L}_\epsilon(F,\Sigma)
\leq
\frac{1}{2\epsilon}
\int_\Sigma
\chi\left(\frac{F(\sigma)}{\epsilon}\right)
|\nabla F(\sigma)|d\sigma
\leq
\frac{1}{2\epsilon}
\int_{\mathbb{T}^3}
\chi\left(\frac{F(x)}{\epsilon}\right)
|\nabla F(x)|dx.
\end{equation}
In \cite[Lemma 3.2]{rudwi2}, it was shown that the quantity
\begin{equation*}
Z_\epsilon(F):=\frac{1}{2\epsilon}
\int_{\mathbb{T}^3}
\chi\left(\frac{F(x)}{\epsilon}\right)
|\nabla F(x)|dx
\end{equation*}
satisfies the uniform bound
\begin{equation*}
Z_\epsilon(F)
\leq
18\sqrt{m}.
\end{equation*}
Substituting the latter bound into \eqref{unifbddeqn}, we obtain the claim of the present lemma. 
\end{proof}

\begin{proof}[Proof of Lemma \ref{lecovmat}]
We show a few of the computations for the covariance matrix $\Phi$, the remaining being similar or following immediately from the symmetry of $\Phi$. Firstly,
\begin{align*}
\phi_{23}:=\mathbb{E}[F(\sigma')a_1(\sigma)]&=\mathbb{E}[F(\sigma')((\nabla F)_1-\langle\nabla F,\overrightarrow{n}\rangle n_1)]
\\&=\langle D,(n_2^2+n_3^2,-n_1n_2,-n_1n_3)\rangle,
\end{align*}
with $D$ as in Definition \ref{defdh}. By a similar computation,
\begin{equation*}
\phi_{24}:=\mathbb{E}[F(\sigma')a_2(\sigma)]=\langle D,(-n_1n_2,n_1^2+n_3^2,-n_2n_3)\rangle,
\end{equation*}
so that, recalling the definitions of the matrices $\Omega$ \eqref{omegafull} and $L$ \eqref{matL}, we may write
\begin{equation*}
\begin{pmatrix}
\phi_{23} & \phi_{24}
\end{pmatrix}=D\Omega L.
\end{equation*}

Since $L^T\Omega L$ is the upper left $2\times 2$ block of $\Omega$, one has
\begin{equation*}
\begin{pmatrix}
\phi_{33} & \phi_{34} \\ \phi_{43} & \phi_{44}
\end{pmatrix}=ML^T\Omega L.
\end{equation*}
Further,
\begin{align*}
\phi_{36}:&=\mathbb{E}[a_1(\sigma)a_2(\sigma')]\\&=\mathbb{E}[((\nabla F)_1-\langle\nabla F,\overrightarrow{n}\rangle n_1)(\sigma)\cdot((\nabla F)_2-\langle\nabla F,\overrightarrow{n}\rangle n_2)(\sigma')]
\\&=H_{11}(-n'_1n'_2)(n_2^2+n_3^2)
+H_{12}((n_2^2+n_3^2)({n'_1}^2+{n'_3}^2)+n_1n'_1n_2n'_2)
\\&+H_{13}((-n'_2n'_3)(n_2^2+n_3^2)+n_1n'_1n'_2n_3)
+H_{22}(-n_1n_2)({n'_1}^2+{n'_3}^2)
\\&+H_{23}((-n_1n_3)({n'_1}^2+{n'_3}^2)+n_1n_2n'_2n'_3)
+H_{33}(n_1n'_2n_3n'_3).
\end{align*}
Similarly, one computes $\phi_{35},\phi_{45},\phi_{46}$ to see that
\begin{equation*}
\begin{pmatrix}
\phi_{35} & \phi_{36} \\ \phi_{45} & \phi_{46}
\end{pmatrix}=-L^T\Omega H\Omega'L.
\end{equation*}
\end{proof}

\begin{proof}[Proof of Lemma \ref{XY}]
	The matrix
	\begin{equation*}
	I_4+
	\begin{pmatrix}
	X & Y \\ Y' & X'
	\end{pmatrix}
	\end{equation*}
	is a covariance matrix, hence, by Cauchy-Schwartz, to show \eqref{XYl1} it will suffice to bound the diagonal entries of $X,X'$. In turn, this may be done by noting that the diagonal entries of a covariance matrix are positive, and that the entries of $X,X'$ are negative (see Definition \ref{XYdef}).
\end{proof}

Next, we compute the Taylor expansion of a perturbed Gaussian covariance matrix, establishing Lemma \ref{lemma5.1}. We employ Berry's elegant method \cite{berry2}, rather than computing the various derivatives by brute force, which would result in a longer computation.
\begin{proof}[Proof of Lemma \ref{lemma5.1}]
Here we outline the main steps and refer the reader to the more detailed \cite[Lemma 5.1]{krkuwi} and \cite[Lemma 5.8]{benmaf}. We begin by writing
\begin{equation}
\label{intxi}
\mathbb{E}[|(\hat{W_1},\hat{W_2})|\cdot|(\hat{W_3},\hat{W_4})|]=
\frac{1}{2\pi}\iint_{\mathbb{R}_+^2}
\xi(t,s)
\frac{dtds}{(ts)^{3/2}}
\end{equation}
with
\begin{equation*}
\xi(t,s):=(\eta(0,0)-\eta(t,0)-\eta(0,s)+\eta(t,s))
\end{equation*}
and
\begin{gather*}
\eta_{X,X',Y,Y'}(t,s)=
\frac{1}{\sqrt{\det{(I_4+J)}}}
\end{gather*}
where
\begin{equation*}
I_4+J=
\begin{pmatrix}
(1+t)I_2+tX & \sqrt{ts}Y \\ \sqrt{ts}Y' & (1+s)I_2+sX'
\end{pmatrix}.
\end{equation*}

We perform the Taylor expansion of $\det(I_4+J)^{-1/2}$ using the formula for the determinant of a block matrix:
\begin{multline*}
\det(I_4+J)^{-1/2}
\\=\det((1+t)I_2+tX)^{-1/2}
\cdot
\det[(1+s)I_2+sX'-tsY'((1+t)I_2+tX)^{-1}Y]^{-1/2}
\end{multline*}
so that \footnote{All error terms in \eqref{eta} are controlled by $X^2+X'^2+Y^4+Y'^4$: for instance, $YXY'\ll\max\{XY^2, XY'^2\}\ll\max\{X^2, Y^4, Y'^4\}$.}
\begin{multline}
\label{eta}
\eta_{X,X',Y,Y'}(t,s)=\frac{1}{(1+t)(1+s)}
-\frac{1}{2}\frac{t}{(1+t)^2(1+s)}tr(X)
\\-\frac{1}{2}\frac{s}{(1+t)(1+s)^2}tr(X')
+\frac{1}{2}\frac{ts}{(1+t)^2(1+s)^2}tr(Y'Y)
\\+O(X^2+X'^2+Y^4+Y'^4).
\end{multline}
This leads to
\begin{align}
\label{xi}
\notag
\xi_{X,X',Y,Y'}(t,s)
&=\frac{ts}{(1+t)(1+s)}
+\frac{1}{2}\frac{ts}{(1+t)^2(1+s)}tr(X)
\\\notag&+\frac{1}{2}\frac{ts}{(1+t)(1+s)^2}tr(X')
+\frac{1}{2}\frac{ts}{(1+t)^2(1+s)^2}tr(Y'Y)
\\&+O\left(\min(t,1)\min(s,1)(X^2+X'^2+Y^4+Y'^4)\right),
\end{align}
where we improved the error term in \eqref{eta} by noting that $\xi(t,s)=O(ts)$. On inserting \eqref{xi} into \eqref{intxi} and integrating we obtain
\begin{multline*}
\mathbb{E}[|(\hat{W_1},\hat{W_2})|\cdot|(\hat{W_3},\hat{W_4})|]
\\=\frac{\pi}{2}\left(1+\frac{tr(X)}{4}+\frac{tr(X')}{4}+\frac{tr(Y'Y)}{8}\right)
+O(X^2+X'^2+Y^4+Y'^4).
\end{multline*}
\end{proof}
We conclude the appendix with the following proof.
\begin{proof}[Proof of Lemma \ref{bdsingl}]
We begin by invoking Lemma \ref{leK2} to write
\begin{equation}
\label{K2te}
\tilde{K}_{2;\Sigma}=MK_{2;\Sigma}=\frac{M}{2\pi\sqrt{1-r^2}}\mathbb{E}[|(\hat{W_1},\hat{W_2})|\cdot|(\hat{W_3},\hat{W_4})|]
\end{equation}
where
\begin{equation*} \hat{W}=(\hat{W_1},\hat{W_2},\hat{W_3},\hat{W_4})\sim\mathcal{N}(0,\hat{\Theta}),
\end{equation*}
and $\hat{\Theta}$ is given by \eqref{hattheta}. As $\sigma'\to\sigma$ the expectation in \eqref{K2te} converges to a positive constant, so that on each diagonal region $V_{i,i}$
\begin{equation*}
\tilde{K}_{2;\Sigma}\asymp\frac{m}{\sqrt{1-r^2}}\asymp\frac{\sqrt{m}}{|\sigma-\sigma'|}
\end{equation*}
\footnote{The notation $f\asymp g$ means $g\ll f\ll g$.} via a Taylor expansion. Therefore,
\begin{equation}
\label{K2bound}
\iint_{V_{i,i}}\tilde{K}_{2;\Sigma}d\sigma d\sigma'\asymp\sqrt{m}\iint_{V_{i,i}}\frac{d\sigma d\sigma'}{|\sigma-\sigma'|}\ll_\Sigma\sqrt{m}\delta^3\asymp\frac{1}{m}.
\end{equation}
By Kac-Rice and the Cauchy-Schwartz inequality,
\begin{multline}
\label{cs}
\left|\iint_{V_{i,j}}\tilde{K}_{2;\Sigma}d\sigma d\sigma'\right|=\left|\mathbb{E}[\mathcal{L}_i\mathcal{L}_j]\right|\leq\sqrt{\mathbb{E}[\mathcal{L}_i^2]\mathbb{E}[\mathcal{L}_j^2]}\\=\left(\iint_{V_{i,i}}\tilde{K}_{2;\Sigma}d\sigma d\sigma'\iint_{V_{j,j}}\tilde{K}_{2;\Sigma}d\sigma d\sigma'\right)^{1/2}
\end{multline}
where
\begin{equation*}
\mathcal{L}_i:=h_1(\mathcal{A}_F \cap \Sigma_i)
\end{equation*}
(recall the notation $h_1$ for Hausdorff measure). By \eqref{K2bound} and \eqref{cs} one obtains
\begin{equation*}
\iint_{S}\tilde{K}_{2;\Sigma}(\sigma,\sigma')d\sigma d\sigma'
\ll
\sum_{V_{i,j} \text{ sing}} \frac{1}{m}
\asymp
\frac{1}{m}\cdot\frac{\text{meas}(S)}{\delta^4}
\ll
m\cdot\mathcal{R}_4(m),
\end{equation*}
where we also applied Lemma \ref{Sbound}. The proof of Lemma \ref{bdsingl} is complete.
\end{proof}

\addcontentsline{toc}{section}{References}
\bibliographystyle{plain}
\bibliography{bibfile}

\Addresses

\end{document}